\documentclass[reqno,11pt]{amsart}

\usepackage{amsmath}
\usepackage{amsthm}
\usepackage{amssymb}
\usepackage{amsbsy}
\usepackage{amsfonts}
\usepackage{amstext}
\usepackage{amscd}
\usepackage{tikz}
\usepackage{graphicx}
\usepackage{verbatim}
\usepackage{bm}
\usepackage{commath,mathtools,setspace}
\usepackage{paralist}
\usepackage{extarrows}
\usepackage{color}
\usepackage{stmaryrd}

\usepackage{babel}

\usepackage[T1]{fontenc}
\usepackage{graphicx}
\usepackage{blindtext} % Package to generate dummy text throughout this template 
\usepackage{geometry}
\usepackage[colorlinks=true,linkcolor=blue,citecolor=blue]{hyperref}
\usepackage[colorinlistoftodos,prependcaption,textsize=small]{todonotes}

\definecolor{red}{RGB}{255,0,0}
\definecolor{green}{RGB}{0,100,0}
\definecolor{blue}{RGB}{0,0,255}

\newtheorem{theorem}{Theorem}[section]
\newtheorem{lemma}[theorem]{Lemma}

\newtheorem{corollary}[theorem]{Corollary}
\newtheorem{proposition}[theorem]{Proposition}
\newtheorem{notation}[theorem]{Notation}

\theoremstyle{remark}
\newtheorem{remark}[theorem]{Remark}
\newtheorem{definition}[theorem]{Definition}
\newtheorem{example}[theorem]{Example}

\newcommand{\R}{\mathbb{R}}

\newcommand{\pols}{\mathcal{P}}

\newcommand{\cc}{\mathbb{C}}

\newcommand{\ee}{\mathbb{E}}

\newcommand{\rr}{\mathbb{R}}

\newcommand{\MM}{\mathcal{M}}

\newcommand{\PP}{\mathcal{P}}

\newcommand{\mfp}{\mathfrak{p}}

\newcommand{\dil}[1]{\mbox{Dil}_{#1}}
\newcommand{\shift}[1]{\mbox{Shift}_{#1}}

\newcommand{\diff}[2]{\partial^{#1|#2}\, }  %Derivative of polynomials

\newcommand{\falling}[2]{\left(#1\right)^{\underline{#2}}}

\newmuskip\pFqmuskip

\newcommand*\pFqN[6][8]{%
  \begingroup % only local assignments
  \pFqmuskip=#1mu\relax
  \mathcode`\,=\string"8000
  \begingroup\lccode`\~=`\,
  \lowercase{\endgroup\let~}\pFqcomma
  {}_{#2}F_{#3}{\left(\genfrac..{0pt}{}{#4}{#5};#6\right)}%
  \endgroup
}
\newcommand{\pFqcomma}{\mskip\pFqmuskip}

\newcommand*\HGP[3]{%
 \mathcal{H}_{#1}{\left[\genfrac..{0pt}{1}{#2}{#3}\right]}%
}

\newcommand*\SRM[2]{%
 \rho{\left[\genfrac..{0pt}{1}{#1}{#2}\right]}%
}

\newcommand{\polar}[1]{D_{#1}}
\newcommand{\polars}[3]{D_{#1}^{#2|#3}}
\newcommand{\freepower}[1]{F^{#1}}
\newcommand{\pfreepower}[2]{F_{#1}^{#2}}

\newcommand{\mobr}{\text{Möb}(\hat{\rr})}
\newcommand{\mobc}{\text{Möb}(\hat{\cc})}

\newcommand{\meas}[1]{\mu \left\llbracket #1 \right\rrbracket}

\newcommand{\pol}{\mathcal{P}}
\newcommand{\polreal}[1]{\mathcal{P}_{#1}(\mathbb{R})}

\newcommand{\coef}[2]{\mathsf{e}_{#1}\left(#2\right)}

\title[Root distribution under repeated polar differentiation]{Asymptotic root distribution of polynomials under repeated polar differentiation}
\author{Daniel Perales}

\address{Daniel Perales: Department of Mathematics, University of Notre Dame, IN, USA}

\email{dperale2@nd.edu}

\author{Zhiyuan Yang}

\address{Zhiyuan Yang: Department of Mathematics, Texas A\&M University, TX, USA}

\email{zhiyuanyang@tamu.edu}

\begin{document}

\begin{abstract}
    Given a sequence of real rooted polynomials $\{p_n\}_{n\geq 1}$ with a fixed asymptotic root distribution, we study the asymptotic root distribution of the repeated polar derivatives of this sequence. This limiting distribution can be seen as the result of fractional free convolution and pushforward maps along Möbius transforms for distributions. This new family of operations on measures forms a semigroup and satisfy some other nice properties. Using the fact that polar derivatives commute with one another, we obtain a non-trivial commutation relation between these new operations. We also study a notion of polar free infinite divisibility and construct Belinschi-Nica type semigroups. Finally, we provide some interesting examples of distributions that behave nicely with respect to these new operations, including the Marchenko-Pastur and the Cauchy distributions.
\end{abstract}

\maketitle

\setcounter{tocdepth}{1}
\tableofcontents

%%%%%%%%%%%%%%%%%%%%%%%%%%%%%%%%%%%%%%%%%%%%%%%%%%%%%%%%%%%%%%%%%%%%%%%%%%
\section{Introduction}
%%%%%%%%%%%%%%%%%%%%%%%%%%%%%%%%%%%%%%%%%%%%%%%%%%%%%%%%%%%%%%%%%%%%%%%%%%

\subsection{Free fractional convolution and repeated differentiation}

The notion of free convolution $ \mu\boxplus \nu $ of two distributions $\mu,\nu$ on $\rr$ was first introduced in \cite{voiculescu1986addition} for distributions with compact support, and later generalized to arbitrary distributions on $\rr$ in \cite{bercovici1993free}. It corresponds to the distribution of $a+b$ where $a$ and $b$ are free and have distribution $\mu$ and $\nu$, respectively. For each integer $n\geq 1$,  the distribution 
$$ \mu^{\boxplus n}:= \underbrace{\mu\boxplus \cdots \boxplus \mu}_{n} $$ is called the free (additive) convolution power of $ \mu$. It is shown in \cite{bercovici1995superconvergence, nica1996multiplicative} that this notion can be extended to a fractional free convolution power $ \mu^{\boxplus t} $  for an arbitrary real number $ t\geq 1 $, forming a semigroup of distributions: for $t,s\geq 1$, $\mu^{\boxplus t}\boxplus \mu^{\boxplus s} = \mu^{\boxplus (s+t)}$.
Recently in \cite{shlyakhtenko2022fractional}, Shlyakhtenko and Tao gave a PDE characterization of free fractional convolution powers.

The study of asymptotic root distributions of polynomials under repeated differentiation has seen a surge of interest in recent years. Given a polynomial $p$ of degree $n$ with roots $\lambda_1,\dots, \lambda_n$, its root distribution is the measure with an atom of size $\tfrac{1}{n}$ at each root:
\begin{equation}
\label{eq:def.meas}
\meas{p}= \frac{1}{n}\sum_{j=1}^n \lambda_j.
\end{equation}

In \cite{steinerberger2019nonlocal}, Steinerberger studied for a fixed $s\in(0,1)$, the behavior of differentiating $\lfloor sn\rfloor$ times a polynomial $p_n$ of degree $n$ and looking at the limiting behavior of the root distribution $\meas{p_n^{(\lfloor sn\rfloor)}}$ when the degree $n$ tends to infinity. Steinerberger derived a PDE of the density function of the asymptotic root distributions that was precisely the same PDE that later appeared in \cite{shlyakhtenko2022fractional} describing the free fractional convolution. Thus, in \cite{steinerberger2023free} Steinerberger formally derived that if $ (p_n)_{n\geq 1} $ is a sequence of real rooted polynomial where $p_n$ has degree $n$ and $\meas{p_n}$ converges weakly to a fixed distribution $\nu$ on $\rr$, then for a $t\geq 1$, 
$$\meas{ p_n^{ (\lfloor (1-1/t)n\rfloor) }} \text{ converges weakly to }  \dil{1/t} \mu^{ \boxplus t } \text{ as } n\to \infty,$$ where $\dil{1/t}$ stands for a dilation of $1/t$. This was rigorously proved in \cite{hoskins2023dynamics} for compactly supported distributions. Another proof was given in \cite{arizmendi2023finite}, using finite free probability techniques. Later the proof was streamlined in \cite{arizmendi2024s} and the result was generalized to allow distributions with possibly unbounded support. Finally, the result was generalized further in a recent work \cite{jalowy2025zeros}, where the roots are allowed to tend to infinity and the limiting measure can have an atom at $\infty$. 

Beyond the real roots scenario, there are results on root distributions under repeated differentiation under varied assumptions on the polynomials: radially symmetric roots \cite{hoskins2023dynamics}; roots on the unit circle \cite{kabluchko2021repeated}; random coefficients \cite{hall2023roots,campbell2024fractional}; and random roots \cite{hoskins2022semicircle,arizmendi2025critical}. There are also studies on the effect on roots after repeatedly applying other differential operators, see for instance the work of \cite{campbell2024free} in connection to free rectangular convolution.

%%%%%%%%%%%%%%%%%%%%%%%%%%%%%%%%%%%%%%%%%%%%%%%%%%%%%%%%%%%%%%%%%%%%%%%%%%
\subsection{Our contribution}
%%%%%%%%%%%%%%%%%%%%%%%%%%%%%%%%%%%%%%%%%%%%%%%%%%%%%%%%%%%%%%%%%%%%%%%%%%
Instead of focusing on the usual derivative of polynomials, we will consider the polar derivative, a natural generalization of the classical derivative for polynomials with respect to a fixed point in the real line. Our goal is to study the limiting behavior of roots of polynomials after repeated polar differentiation. Given $a\in \cc$, and a polynomial $p$ of degree $n$, the polar derivative $\polar{a}p$ of $p$ with respect to $a$ can be defined as the polynomial 
$$\polar{a}p(x):= np(x)-(x-a)p'(x).$$
When $a$ tends to $\infty$, then $\polar{a}$ tends to the usual differentiation $\partial$, thus we set $\polar{\infty} p: = \partial p$. 

Our approach to understand repeated polar differentiation is to use the fact that polar derivatives are intertwined by Möbius transforms. In particular, we can view polar differentiation as usual differentiation conjugated by an appropriate Möbius transform. Recall that given $\alpha,\beta,\gamma,\delta\in \cc$, a Möbius transform is a bijective map $T$ from $\hat{\cc}:= \cc\cup \{\infty\}$ into itself, of the form 
$$ T(z) = \frac{\alpha z+\beta}{\gamma z+\delta}.$$
Given a polynomial $p$ with roots $\lambda_1,\dots,\lambda_n$, and a Möbius transform $T$, we denote by $T_*p$ the polynomial with roots $T(\lambda_1),\dots,T(\lambda_n)$. One can check that if $a\neq \infty$ and $T$ is a Möbius transform that maps $T(a)=\infty$, then
$$D_a= (T^{-1})_* \partial T_*.$$

Since repeated differentiation in the limit approaches (dilated) fractional free convolution powers $ \freepower{t}\mu:= \dil{\tfrac{1}{t}}\mu^{\boxplus t} $, and Möbius transforms are continuous almost everywhere, it is natural to expect that the asymptotic root distribution under repeated polar differentiation can be described by free fractional convolution after a conjugation by an appropriate Möbius transform. If $T$ is a Möbius transform that maps $T(a)=\infty$, say $T(z)=\frac{1}{z-a}$, we denote by $\pfreepower{a}{t}$ the operation on measures defined as
\[  \pfreepower{a}{t}\mu := T^{-1}_*\freepower{t}T_*\mu = T^{-1}_* ( \dil{\tfrac{1}{t}}(T_*\mu)^{\boxplus t} ) , \]
where $ T_*{\mu} $ represents the pushforward measure along $ T $. Interestingly enough, $\pfreepower{0}{t}\mu$ can be interpreted as doing the multiplicative free convolution of $\mu$ times a measure with atoms at 1 and $\infty$, see
Section \ref{sec:connection.to.free.mult} for more details.

With this notation in hand we are ready to state our first main result, that studies the limiting behavior of roots after repeated polar differentiation.
We use $\polars{a}{k}{n}$ to denote the operation of taking an $a$-polar derivative $n-k$ times.

\begin{theorem}[Asymptotic root distribution under repeated polar differentiation]
\label{thm:asymptotic.polar.diff.intro}
Fix a sequence of real rooted polynomials $p_j$ of increasing formal degrees $n_j$ such that $\meas{p_j}$ converges weakly to a probability measure $\mu \in \MM({\rr})$. Assume that for all $j$, the polynomial $p_j$ has less than $ m_j $ roots equal to $a$. Let $t\geq 1$ be a parameter and consider a sequence $m_j$ such that $\lim_{j\to\infty}\frac{n_j}{m_j}=t$. Then
we have the following weak limit:
$$\lim_{j\to\infty} \meas{\polars{a}{m_j}{n_j} p_j}=  \pfreepower{a}{t}\mu.$$
\end{theorem}

Once this is settled, one can try to translate some properties of polar derivatives in terms of the operation $ \pfreepower{a}{t}$. For example, it is not hard to check that polar derivatives with respect to distinct points commute with each other:
$$\polar{a}\polar{b}=\polar{b}\polar{a}.$$
While this relation is basic from the perspective of polynomials, we can already use it to derive the following non-trivial commutation relation for the maps $\pfreepower{a}{t}$ and $\pfreepower{b}{t}$ on measures.

\begin{theorem}\label{thm:commutation.intro}
Let $ s,t\geq 1$, and $a,b\in \rr\cup \{\infty\}$. Let also $s',t' >1$ be the constants determined by the relations $ st=s't' $, $ s+s' = 1+st $. Then 
$$ \pfreepower{a}{s}\pfreepower{b}{t} \mu =\pfreepower{b}{s'}\pfreepower{a}{t'}\mu \in \MM(\rr/\{a,b\}), \qquad \forall \mu \in \MM(\rr/\{a,b\}). $$
\end{theorem}

\begin{remark}
    We note that the relation in Theorem \ref{thm:commutation.intro} is very similar to the commutation relation between the Boolean convolution power and the free convolution power in \cite{belinschi2008remarkable}:
    $$ \left(\mu^{\boxplus t}\right)^{\uplus s } = \left(\mu^{\uplus t'}\right)^{\boxplus s' }, $$
    where $ \mu^{\uplus t} $ is the Boolean convolution power of $\mu$, and $s',t'\geq 1$ are constants determined by $ st=s't' $ and $t+t'=1+st$. However, there is a notable difference in the index conventions:
    \begin{itemize}
    \item In the Belinschi-Nica identity, the sum of the inner indices satisfies the relation $t+t'=1+st $.
    \item Whereas in our setting, the sum of the outer indices satisfies the relation $ s+s'=1+st $.    
    \end{itemize}
\end{remark}

\subsection{Application and examples}

As a first example, consider the Laguerre polynomial of degree $n$ and parameter $\lambda$,
$$\HGP{n}{\lambda}{\cdot}:=\sum_{k=0}^n x^{n-k} (-1)^k \binom{n}{k} \falling{n\lambda}{k}.$$

It is known that 0-polar differentiation preserves Laguerre polynomials. In Section \ref{sec:examples}, we will check that  differentiating $n-m$ times yields
$$\polars{0}{m}{n}\HGP{n}{\lambda}{\cdot}=c \, \HGP{m}{\frac{n}{m}(\lambda-1)+1}{\cdot},$$ 
where $c$ is just some constant.

By letting $n$ and $m$ tend to $\infty$ while $\frac{n}{m}$ tends to $t$, Theorem \ref{thm:asymptotic.polar.diff.intro} asserts that a similar relation holds for the free Poisson (or Marchenko-Pastur) distribution: 
$$
F_0^t\left( \pi_\lambda \right) = \dil{\frac{1}{t}}\,  \pi_{t\lambda-t+1}.
$$
We are not aware of a similar result in the literature. Notice that the previous relation means that the family of free Poisson distributions $(\pi_\lambda)_{\lambda > 1}$ is invariant under both semigroups $\freepower{t}$ and $ \pfreepower{0}{t}$. Where the former is a well-known fact in free probability: $\freepower{t} \pi_\lambda= \dil{1/t} \pi_{t\lambda}$.\\

A similar relation holds for the much larger class of \emph{hypergeometric polynomials}. In the limit, this relation translate into non-trivial relation for the class of \emph{$S$-rational measures}, see Section \ref{sec:examples} for details.\\ 

Another interesting fact is that the Cauchy distribution $\nu$ with density $$ d\nu := \frac{1}{\pi}\frac{1}{1+x^2}dx.$$ 
is invariant under $\pfreepower{a}{t}$ for all $a\in \hat{\rr}$. Namely,
$$ \pfreepower{a}{t}\, \nu = \nu \qquad \text{for } t\geq 1. $$ 

The analogue of the Cauchy distribution $\nu$ in the framework of polynomials was devised recently by Campbell \cite[Theorem 2.4]{campbell2024free}. In fact, it is observed that $\nu$ can actually be realized as the asymptotic root distribution of the sequence $ \{C_n\}_{n\geq 1} $, where
$$ C_n(x) := \cos(\partial)x^n =\sum_{k=0}^{\left\lfloor n/2 \right\rfloor} (-1)^k \binom{n}{2k} x^{n-2k}. $$
Notice that $ \{C_n\}_{n\geq 1} $ is an Appell sequence, so it is invariant under differentiation: $ C_{n-1}=\partial C_n $. Moreover, considering $C_n$ as a polynomial with formal degree $n$, up to a constant, we have the equality 
$$ \polar{0}^{2k}C_n = C_{n-2k}=\partial^{2k}C_n.$$

In general, for $0\neq a\neq \infty$, the $a$-polar derivatives of $C_n$ are much more involved and $\polars{a}{k}{n}C_n$ does not have a straightforward relation with $ C_{n-k} $. However, since $ \pfreepower{a}{t}\nu =\nu$, Theorem \ref{thm:asymptotic.polar.diff.intro} guarantees that asymptotically they still converges to $ \nu $.\\

\textbf{Organization of the paper.} Besides this introduction we have seven more sections. In Section \ref{sec:prelim}, we first set up the notation for measures and polynomials and recall the relations between repeated derivatives and fractional free convolution. We then review properties of Möbius transforms and polar derivatives. In Section \ref{sec:polarfreeconvolution}, after extending the definition of fractional free convolution to measures on $ \rr\cup \{\infty\}$, we define the polar fractional free convolution power of distributions and prove Theorem \ref{thm:asymptotic.polar.diff.intro}. We explain the connection between $\pfreepower{0}{t}$ and multiplicative free convolution in Section \ref{sec:connection.to.free.mult}. In Section \ref{sec:pfofcommutationrelation} we provide two proofs of Theorem \ref{thm:commutation.intro}, one using polar differentiation and one using operator models of polar free convolution powers. We study the polar infinite divisible distributions in Section \ref{sec:infdivisibility}, where we also construct a Belinschi-Nica type semigroup for two given points in $\rr\cup \{\infty\}$. In Section \ref{sec:analytical}, we provide a non-rigorous computation of the PDE for the Cauchy transform and $R$-transform of the family of polar free convolution powers. Finally, several examples of polynomials and measures are studied in Section \ref{sec:examples}.

%%%%%%%%%%%%%%%%%%%%%%%%%%%%%%%%%%%%%%%%%%%%%%%%%%%%%%%%%%%%%%%%%%%%%%%%%%
\section{Preliminaries}
\label{sec:prelim}
%%%%%%%%%%%%%%%%%%%%%%%%%%%%%%%%%%%%%%%%%%%%%%%%%%%%%%%%%%%%%%%%%%%%%%%%%%

%%%%%%%%%%%%%%%%%%%%%%%%%%%%%%%%%%%%%%%%%%%%%%%%%%%%%%%%%%%%%%%%%%%%%%%%%%
\subsection{Möbius transform}
We will denote by $\hat{\cc}:=\cc \cup \{\infty\}$ the extended complex plane. Given $a,b,c,d\in\cc$, a Möbius transform is an invertible map $T: \hat{\cc}\to \hat{\cc}$ which has the form
$$ T(z) = \frac{az+b}{cz+d}.$$
The inverse of $T$ is also a Möbius transform 
$$T^{-1}(z)= \frac{dz-b}{-cz+a}.$$
Notice that $T$ is almost an invertible map on $\cc$ except at two points:
$$T(-d/c)=\infty \qquad \text{and} \qquad T(\infty) = a/c.$$

Another important well-known fact that will be useful later is that the composition of Möbius transform is another Möbius transform. We denote by $\mobc$ the group of all Möbius transforms.

We are mainly interested in transformations that preserve $\hat{\rr}:=\rr\cup\{\infty\}$. We denote by $\mobr$ the subgroup of $\mobc$ satisfying $T(\hat{\rr})=\hat{\rr}$. Notice that if $T\in\mobr$, then we can always choose $a,b,c,d\in \rr$ such that $ T(z) = \frac{az+b}{cz+d}$.

Particular cases of Möbius transforms are
\begin{itemize}
    \item \textbf{Dilations.} That we denote by $\dil{a}$ and are defined as $\dil{a}(z)=az$ for $a\in\rr$.
    \item \textbf{Shifts.} That we denote by $\shift{b}$ and are defined as $\shift{b}(z)=z+b$ for $b\in\rr$.
\end{itemize}

\begin{remark}
\label{rem:Mobius.fix.infty}
If $T$ is a Möbius transform preserving $\infty$, namely $T(\infty)=\infty$, then $c=0$ and $T$ has the form $ T(z) = \frac{a}{d}z+\frac{b}{d} $. Thus, $T$ can be seen as a composition of a dilation and a shift:
$$T= \dil{a/d}\circ\shift{b/a} = \shift{b/d}\circ \dil{a/d}.$$
\end{remark}

%%%%%%%%%%%%%%%%%%%%%%%%%%%%%%%%%%%%%%%%%%%%%%%%%%%%%%%%%%%%%%%%%%%%%%%%%%

\subsection{Polynomials and polar derivatives}

In this section we study the basics of polar differentiation that can be understood as generalization of the usual differentiation. 

The study of polar derivatives is a very old topic in the literature. Here we will do a quick survey following the books by Marden \cite{marden1949geometry}, and by P{\'o}lya and Szeg{\"o} \cite{polya2012problems}.

\begin{notation}
    Let $\pol_n(\hat{\cc})$ be the set of polynomials
    $$ \pol_n(\hat{\cc}) = \{ p(z): p(z)= a_0+a_1z+\cdots + a_nz^n,\;a_i\in \cc \}, $$
    note that here $ a_n $ is not required to be different from zero. For $p\in \pol_n(\hat{\cc})$, we say that:
    \begin{enumerate}[a)]
        \item The \textbf{formal degree} of $p$ is $n$.
        \item The \textbf{(precise) degree} of $p$ is $ 
\deg p := \max\{m: a_m\neq 0\} $, where we set $ \deg p:= -\infty $ if $p \equiv 0$.
    \end{enumerate}
    
    For a $p$ with precise degree $0\leq k<n$, we say that $p$ has $(n-k)$-fold zeros at $\infty$. With this notation, every nonzero polynomial $ p\in \pol_n(\hat{\cc}) $ has precisely $ n $ roots in $ \hat{\cc} = \cc\cup \{\infty\} $. Moreover, $p$ is determined up to a constant of proportion by its $n$ roots.
    
    For a subset $K\subset \hat{\cc}$, we denote by $\pol_n(K)$ the subset of polynomials in $ \pol_n(\hat{\cc})$ whose $n$ roots lie entirely in $K$. We are particularly interested in the set $\pol_n(\hat{\rr})$ of polynomials with all roots real or at $\infty$.
\end{notation}

The notion of formal degree aligns naturally with the projective nature of polynomials. Specifically, the roots of a polynomial $ p(z)= a_0+a_1z+\cdots + a_nz^n $ can be interpreted in projective coordinates as the $n$ roots of the homogeneous polynomial of degree $n$:
\begin{notation}
    The projective form of the nonzero polynomial $p(z) = a_0+a_1z+\cdots + a_nz^n$ is its homogeneous counterpart, considered as functions on the projective line $\mathbb{P}^1(\cc) = \hat{\cc}$

    $$ \hat{p}(z_1,z_2) = a_0z_1^n+a_1 z_2z_1^{n-1}+\cdots+ a_nz_2^n. $$
\end{notation}

Under this projective perspective, differentiation is naturally expressed using the polar derivative. Given a homogeneous polynomial $\hat{p}(z_1,z_2)$ of degree $n$, its polar derivative at a point $ (\alpha_1:\alpha_2)\in \mathbb{P}^1(\cc) = \hat{\cc}$ is defined as
$$ \partial_{ (\alpha_1:\alpha_2) }\hat{p} := \alpha_1\frac{\partial\hat{p}}{\partial z_1}+ \alpha_2\frac{\partial\hat{p}}{\partial z_2} .$$
In particular, when evaluated at $\infty = (0:1)$, the polar derivative is
$$ \partial_{ (0:1) }\hat{p}(z_1,z_2) := \frac{\partial\hat{p}}{\partial z_2}(z_1,z_2), $$
which corresponds precisely to the usual derivative $p'$ of a polynomial $p\in \pol_n(\hat{\cc})$. For general $ (\alpha_1:\alpha_2) \neq \infty $, we have
\begin{align*}
    \partial_{ (\alpha_1:\alpha_2)}\hat{p}(z_1,z_2) &= \alpha_1 \sum_{k=0}^{n-1} (n-k) a_k z_1^{n-k-1} z_2^{k} + \alpha_2 \sum_{k=1}^{n} k a_k z_1^{n-k} z_2^{k-1}.
\end{align*}

Setting \( z = z_2/z_1 \), we obtain the polynomial in \( z \):
\[
\partial_{(\alpha_1:\alpha_2)} \hat{p}(1, z) = \sum_{k=0}^{n-1} (n-k) a_k z^{k} \alpha_1 + \sum_{k=1}^{n} k a_k z^{k-1} \alpha_2.
\]
Rewriting this expression:
\[
\alpha_1^{-1}\partial_{(\alpha_1:\alpha_2)} \hat{p}(1, z) = n p(z) - (z - \alpha) p'(z),
\]
where \( \alpha = \frac{\alpha_2}{\alpha_1} \).

With this intuition, we define the polar derivative for a $p\in \pol_n(\hat{\cc})$ as follows:

\begin{definition}[Polar derivative]
 Given $\alpha\in \cc$, if $ \alpha\neq \infty $, the polar derivative at $\alpha$ is the derivation map $ \polar{\alpha}:\pol_n(\hat{\cc})\to  \pol_{n-1}(\hat{\cc}) $ satisfying
$$ \polar{\alpha}p(x) := np(x) - (x-\alpha)p'(x),\quad \text{for all } p\in \pol_n(\hat{\cc}).  $$
On the other hand, the polar derivative at $\alpha = \infty $, is defined as $ \polar{\infty} = \partial $ the usual differentiation operation.

We will denote by
$$ \polars{\alpha}{k}{n}:\pol_n(\hat{\cc})\to  \pol_{k}(\hat{\cc})  $$
the operation of applying the polar derivative $n-k$ times, emphasizing the change in formal degree.

\end{definition}

\begin{example}
A simple case, that will be useful later, is when the polynomial has multiple roots in $\alpha\in \cc$, the point at where we differentiate. Notice that
$$\polar{\alpha}(x-\alpha)^n = n(x-\alpha)^n - (x-\alpha)n(x-\alpha)^{n-1} = 0 \qquad \text{for all } n\geq 1,$$
While if  $ p(x) = (x-\alpha)^kq(x)$ for some polynomial $q(x)$ of formal degree $k$, then
$$ \polar{\alpha} p(x) = (\polar{\alpha}(x-\alpha)^k )q(x)+(x-\alpha)^k \polar{\alpha} q(x) = (x-\alpha)^k \polar{\alpha} q(x).$$
\end{example}

\begin{example}\label{eg:precisedegreeofpolar}
Another case of concern is when $\alpha=\frac{\lambda_1+\dots+\lambda_n}{n} $ is the average of the roots $\lambda_1,\dots, \lambda_n $ of $p$. In this case $p$ has the expansion 
$$ p(x) = x^n - \alpha n x^{n-1}+q(x),$$ for some $q$ of degree $n-2$ or less. Then,
\begin{align*}
 \polar{\alpha} p(x) &= nx^n-n^2\alpha x^{n-1} +nq(x)-(x-\alpha)\left[n x^{n-1} - (n-1)\alpha n x^{n-2}+ q'(x) \right]   \\
 & = nq(x) -(n-1)\alpha^2 x^{n-2}-(x-\alpha)q'(x).
\end{align*} 
In this case,  $\polar{\alpha} p$ has precise degree strictly less than $n-1$ (so it has roots at $\infty$). Otherwise, if $\alpha\neq \frac{1}{n}\sum_{i=1}^n \lambda_i$, then $\polar{\alpha} p$ has precise degree $n-1$.
\end{example}

Notice that for values $ (\alpha_1:\alpha_2),(\beta_1:\beta_2)\in \mathbb{P}^1(\cc) $, the derivatives $ \partial_{ (\alpha_1:\alpha_2)} $ and $ \partial_{ (\beta_1:\beta_2)} $ commute. It follows that polar derivatives commute.

\begin{proposition}
Given $\alpha,\beta\in \hat{\cc}$, we have that $ \polar{\alpha} $ commutes with $ \polar{\beta} $. Namely
    $$ \polar{\alpha}\polar{\beta}p = \polar{\beta}\polar{\alpha}p,\qquad \text{for all }p\in \pol_n(\hat{\cc}).$$
\end{proposition}

Polar derivatives are known to have nice properties under Möbius transformations, which we will explain in detail below.
\begin{notation}
If $p\in \PP_n(\hat{\cc})$, and $ T(z) = \frac{az+b}{cz+d} $ is a Möbius transform with inverse $ T^{-1}(z) = \frac{dz-b}{-cz+a} $, we denote by $T_*p\in \PP_n(\hat{\cc})$ the polynomial 
\begin{equation}
T_*p(x):=(-cx+a)^n p(T^{-1}(x)). 
\end{equation}
\end{notation}
Note that the roots of $T_*p$ are precisely $\{T(\lambda) : \lambda \text{ is a root of }p\}$.

\begin{proposition}\label{prop:FintertwinedbyMobius}
If $\alpha,\beta\in \hat{\cc}$ and $T\in\mobc$ is such that $T(\alpha)=\beta$ then $T_*\circ \polar{\alpha} = \polar{\beta}\circ T_*$ on $ \PP_n(\hat{\cc})$.
\end{proposition}

\begin{proof}
One can directly check this using the definition $ \polar{\alpha}p(z)=np(x) - (x-\alpha)p'(x) $ when $\alpha,\beta\in \cc$ and when one of them is $\infty$. We instead provide here a unified proof using the projective coordinates: $ T_*p $ in the projective coordinates corresponds to
$$ T_*\hat{p}(z_1,z_2) = \hat{p}( -cz_2+az_1,dz_2-bz_1 ). $$
Since $T^{-1}(\beta)=\alpha $, if $ \beta=(\beta_1,\beta_2) $, we can assume that $ \alpha=( -c\beta_2+a\beta_1:d\beta_2-b\beta_1 ) $. Therefore, by the chain rule
$$ \partial_{(\beta_1,\beta_2)}T_*\hat{p}(z_1,z_2) = \left(\partial_{ (-c\beta_2+a\beta_1:d\beta_2-b\beta_1 ) }\hat{p}\right)( -cz_2+az_1,dz_2-bz_1 ), $$
which precisely corresponds to $ T_*\polar{\alpha}p $.
\end{proof}

As a particular case, we can characterize the polar derivative $\polar{\alpha}$ as the map obtained by conjugating the standard derivative with a Möbius transform sending $\alpha$ to $\infty$.

\begin{corollary}
If $T\in\mobc$ is such that $T(\alpha)=\infty$, then on it holds that
$$D_\alpha= T_*^{-1} \partial T_* \qquad \text{on }\PP_n(\hat{\cc}).$$
\end{corollary}

\begin{remark}
Notice that the last result implies that the polar derivative at a real number $\alpha$ preserve real roots, namely, if $\alpha\in \hat{\rr}$, then $ \polar{\alpha} \left(\pol_n(\hat{\rr}) \right) \subset \pol_{n-1}(\hat{\rr}) $. Indeed, the standard derivative preserves real roots and for $\alpha\in \hat{\rr}$ we can always choose $T\in \mobr$ with $T(\alpha)=\infty$.
\end{remark}

\subsection{Interlacing and stochastic order for polar derivatives}
In this subsection, we focus on polynomials with actual real roots (namely the precise degree matches the formal degree).

\begin{notation}[Roots]
Given a polynomial $p\in \pol_n(\cc)$ we denote its \textbf{roots} by 
$$\lambda_1(p), \dots, \lambda_n(p).$$ 
When $p\in\polreal{n}$, we always assume that 
\begin{equation}
\label{eq:zeros}
\lambda_1(p)\leq\lambda_2(p)\leq \dots\leq \lambda_n(p).
\end{equation}

Given a polynomial $p\in\pol_n(\cc)$, its \textbf{empirical root distribution (or zero counting measure)}, is the measure 
\begin{equation}
 \meas{p} :=\frac{1}{n} \sum_{i=1}^n \delta_{\lambda_i(p)}.
\end{equation}
\end{notation}

\begin{notation}[Stochastic order on polynomials]
\label{nota:partial.order}
Given $p,q\in\PP_n(\rr)$ we say that $p$ is dominate by $q$, and denote it $p \ll q$, if the zeros of $p$ are smaller than the corresponding zeros of $q$ in the following sense:
\[\lambda_k(p) \leq \lambda_k(q) \qquad \text{for all }k=1,2,\dots,n.\]
Clearly $\ll$ is a partial order in $\pol_n(\rr)$.
\end{notation}

\begin{definition}[Interlacing] 
Let $p\in \pol_n(\R)$ and $q\in \pol_m(\R)$, with the zeros denoted as in \eqref{eq:zeros}.  
We say that $q$ \textbf{interlaces} $p$, and denote it $p \preccurlyeq q$, if 
\begin{equation} \label{interlacing1}
    m=n \quad \text{and} \quad \lambda_1(p) \leq \lambda_1(q) \leq \lambda_2(p) \leq \lambda_2(q) \leq \cdots \leq  \lambda_n(p) \leq \lambda_n(q),
\end{equation}
or if
\begin{equation} \label{interlacing2}
    m=n-1 \quad \text{and} \quad \lambda_1(p) \leq \lambda_1(q) \leq \lambda_2(p) \leq \lambda_2(q) \leq \cdots \leq  \lambda_{n-1}(p) \leq \lambda_{n-1}(q) \le \lambda_n(p).
\end{equation}
Furthermore, we use the notation $p \prec q$ when all inequalities in \eqref{interlacing1} or \eqref{interlacing2} are strict. Notice that $\prec$ or $\preccurlyeq$ are not order relations, since they lack transitivity.
\end{definition}

\begin{remark}
\label{rem:implication.orders}
It is readily seen that if $p$ and $q$ have the same degree, then $p\prec q$ implies $p \ll q$. 
\end{remark}

\begin{proposition}\label{prop:interlacingPandDp}
    If $p\in \pol_n(\rr)$, and $\lambda_1(p)< a <\tfrac{1}{n}(\lambda_1(p)+\cdots +\lambda_n(p))$, then $p\preccurlyeq (x-a) \polar{a}p$, and if $ \tfrac{1}{n}(\lambda_1(p)+\cdots +\lambda_n(p)) <a < \lambda_n(p)$, then $ (x-a) \polar{a}p \preccurlyeq p$. When $p$ has distinct roots, we can also replace $\preccurlyeq$ by $\prec $.
\end{proposition}
\begin{proof}
    Suppose that $ a >\tfrac{1}{n}(\lambda_1(p)+\cdots +\lambda_n(p)) $. If $ a $ is a root of $p$ with multiplicity $ k $, then $p = (x-a)^kq$ with $q\in \pol_{n-k}({\rr}/\{a\})$. Since $ p=(x-a)^kq \prec (x-a)\polar{a}p=(x-a)^{k+1}\polar{a}q $ if and only if $ q \prec (x-a)\polar{a}q $. Therefore, replacing $p$ with $q$, we may assume without loss of generality that $a$ is not a root of $p$.

    Now, assume that we have
    $$ \lambda_1(p)\leq \cdots \leq \lambda_{j}(p)<a< \lambda_{j+1}(p)\leq \cdots \leq \lambda_n(p). $$
    Let $T = \frac{1}{z-a}$, then $ T_*p $ has roots $ T(\lambda_i(p)) $ such that
    $$ T(\lambda_j(p))\leq \cdots \leq T(\lambda_1(p))<0<T(\lambda_n(p))\leq \cdots \leq T(\lambda_{j+1}(p)), $$
    and therefore there exists a unique root of $ \partial T_*p $ inside $ [T(\lambda_i(p),T(\lambda_{i-1}(p))] $ for each $ i=2,\cdots,j,j+2,\cdots ,n $, $[T(\lambda_1(p)),T(\lambda_n(p))]$, and $ [T(\lambda_i(p)), T(\lambda_{i+1}(p))] $. For the root $x_0$ inside $ [T(\lambda_1(p)),T(\lambda_n(p))] $,
    we have that $ y_0=T^{-1}x_0 = \tfrac{1}{x_0}+a $ is a root of $ 
\polar{a}p = T^{-1}_* \partial Tp $. Since $T^{-1}$ sends roots of $  \partial T_*p $ to roots of $ \polar{a}p $, we can determine the location of roots of $  \polar{a}p$. First, we note that since $ a\neq \tfrac{1}{n}(\lambda_1(p)+\cdots +\lambda_n(p)) $, $ T^{-1}x_0 \neq \infty $ by Example \ref{eg:precisedegreeofpolar}, and in particular, $x_0\neq 0$. So, there are two cases:
\begin{enumerate}[a)]
    \item If $x_0<0$, then $y_0=T^{-1}x_0<a$ and the roots of $ \polar{a}p $ satisfy
    $$ y_0=\lambda_1( \polar{a}p )\leq \lambda_1(p)\leq \lambda_2( \polar{a}p ) \leq \cdots \leq  \lambda_{j}( \polar{a}p )\leq \lambda_{j}(p)<a< \lambda_{j+1}(p)\leq \lambda_{j+1}( \polar{a}p ) \leq\cdots \leq \lambda_n(p). $$
     \item If $x_0>0$, then $y_0=T^{-1}x_0>a$ and the roots of $ \polar{a}p $ satisfy
    $$ \lambda_1(p)\leq \lambda_1( \polar{a}p ) \leq \cdots \leq  \lambda_{j-1}( \polar{a}p )\leq \lambda_{j}(p)<a< \lambda_{j+1}(p)\leq \lambda_j( \polar{a}p ) \leq\cdots \leq \lambda_n(p) \leq \lambda_{n-1}( \polar{a}p )=y_0. $$
\end{enumerate}

Finally, we claim that we must have $ y_0<a $ which then implies the desired statement for $ 
  a >\tfrac{1}{n}(\lambda_1(p)+\cdots +\lambda_n(p))$. Indeed, suppose that $y_0>a$, then it must be in case b), and we have $ y_0\geq \lambda_i(p) $ for all $i=1,\cdots,n$. Now, since $\polar{a}p(y_0) = 0$, we obtain $ 0=np(y_0)-(y_0-a)p'(y_0) $. This implies the contradiction
  $$n= \frac{(y_0-a)p'(y_0)}{p(y_0)} = \sum_{i=1}^n \frac{y_0-a}{y_0-\lambda_i} \leq n \frac{y_0 -a}{y_0 - \tfrac{\lambda_1(p)+\cdots +\lambda_n(p)}{n}} <n, $$
  where we used the Jensen's inequality (we also note that $\polar{a}p(y_0)\neq 0$ for otherwise $  y_0$ should be a repeated root of $p$ which contradicts the inequalities in b) as the multiplicity of $y_0$ for $\polar{a}p$ should be one less than the multiplicity for $p$).

  If $p$ has distinct roots, then one observes that all of the above inequalities are strict and, therefore, $ (x-a)\polar{a}p\prec p$. The proof for $ a <\tfrac{1}{n}(\lambda_1(p)+\cdots +\lambda_n(p)) $ is the same.
\end{proof}

In particular, we also have the following. See also \cite[Lemma 4.2]{fisk2006polynomials} for the case when $a=0$.
\begin{corollary}
    If $a> \lambda_n(p)$ or $a< \lambda_1(p)$, then we have
$$p\preccurlyeq \polar{a}p. $$
\end{corollary}
\begin{proof}
    For this, we simply note that if $ a >\lambda_n(p)$ then also $ a>\lambda_{n-1}(\polar{a}p) $. In particular, if we pick $ T(z) = \frac{1}{a-z} $, then $T_*$ preserves the order of the roots of $ p $ and $ \polar{a}p $, and therefore $ T_*p \preccurlyeq \partial T_*p =T_*\polar{a}p $ implies that $ p\preccurlyeq \polar{a}p $. The cases for $ a<\lambda_1(p) $ is similar.
\end{proof}

\begin{proposition}
\label{prop:polar.interlacing}
    If $a<b<\lambda_{1}(p)$ or $ \lambda_{n}(p)<a<b $, then
$$\polar{b}p\preccurlyeq \polar{a}p.$$
\end{proposition}
\begin{proof}
    We show this for $ a<b<\lambda_{1}(p)$, the other situation is similar. We first prove this for $p$ with distinct roots. We note that the $i$-th root of $\polar{a}p$ is precisely the solution of $\frac{n}{x-a}=\sum_j \frac{1}{x-\lambda_j(p)}$ within the interval $(\lambda_i(p),\lambda_{i+1}(p))$. The right-hand side function is strictly decreasing in each interval $ (\lambda_i(p),\lambda_{i+1}(p))$ and also $ \frac{n}{x-b}>\frac{n}{x-a} $ for all $x\geq\lambda_{1}(p)$. Therefore, one must have $ \lambda_{i}(p)<\lambda_{i}(\polar{a}p)<\lambda_{i}(\polar{b}p)<\lambda_{i+1}(p) $. This implies $ \polar{b}p\prec \polar{a}p $. For the general case, we consider a sequence of simple polynomials $(p_{m})_{m=1}^\infty$ such that $ |\lambda_i(p_m)-\lambda_i(p)|\to 0$ for each $i$. Since $ \polar{b}p_m\prec \polar{a}p_m $, taking the limit $m\to \infty$ we obtain $ \polar{b}p\preccurlyeq \polar{a}p $, as $ \preccurlyeq $ is preserved under limits.
\end{proof}

\begin{lemma}
\label{lem:polar.derivatives.partial.order}
    If $a<b<\lambda_1(p)$ and $k\leq n$, then $ \polar{a}^kp \ll \polar{b}^kp $. The same is true if $ \lambda_{n}(p)<a<b $.
\end{lemma}
\begin{proof}
By repeatedly applying Proposition \ref{prop:polar.interlacing} we obtain
$ \polar{a}^kp \preccurlyeq  \polar{a}^{k-1}\polar{b}p \preccurlyeq\cdots \preccurlyeq\polar{b}^kp $. The conclusion then follows from Remark \ref{rem:implication.orders} and the transitivity of $\ll$.
\end{proof}

%%%%%%%%%%%%%%%%%%%%%%%%%%%%%%%%%%%%%%%%%%%%%%%%%%%%%%%%%%%%%%%%%%%%%%%%%%
\subsection{Measures}

We use $\MM$ to denote the family of all Borel probability measures on $\hat{\cc}$. When we want to specify that the measure vanishes outside a subset $K \subset \hat{\cc}$ we use the notation
\begin{equation}
\label{eq:not.measure.vanishes}
  \MM(K) := \{\mu \in \MM \mid \mu(K^c)=0 \}\subseteq \MM(\hat{\cc}).
\end{equation}
For instance $\MM(\rr)$ is the set of probability measures on the real line, and for a $a\in \rr$, $$ \MM(\rr/\{a\}):= \{\mu \in \MM(\rr)| \mu(\{a\}) = 0\} .$$

We will consider weak convergence on $\MM(\hat{\rr})$ and $ \MM(\rr) $. Recall that a sequence $ (\mu_n)_{n\geq 1} $ of measures on $\rr$ (or $\hat{\rr}$) is said to converge weakly to $\mu $ if for any continuous bounded function $f$ on $ \rr $ (or $\hat{\rr}$), 
$$ \lim_{n\to\infty} \int_\rr f d\mu_n = \int_\rr f d\mu. $$ Equivalently, $ \mu_n $ converge weakly to $\mu$ if and only if $ \lim \mu_n(I) = \mu(I) $ for any interval $I$ (for $\hat{\rr}$, $I$ can include $ \infty $). In particular, the topology of weak convergence on $\MM(\hat{\rr})$, when restricted to $\MM(\rr)$, coincides with the topology of weak convergence on $ \MM(\rr) $.

Throughout this work, we will reqularly apply Möbius transformations to measures. 
\begin{notation}
\label{nota:mobius.measure}
Given a measure $ \mu \in \MM(\hat{\rr})$, and $ T\in\mobr$ we denote by $ T_*\mu $ the pushforward measure $ T_*\mu(E) := \mu( T^{-1}(E) )\in \MM(\hat{\rr}) $. 
\end{notation}

\begin{remark}
\label{rem:Mob.cont}
    Note that $T_*:\MM(\hat{\rr})\to \MM(\hat{\rr})$ is continuous with respect to the weak convergence topology since $T$ is continuous on $ \hat{\rr} $ (as a subspace of $ \hat{\cc} $).
\end{remark}

\begin{notation}[Cauchy transform]
    For a probability measure $\mu \in \MM(\hat{\rr})$, the \emph{Cauchy transform} of $\mu$ is defined by
    \[ G_{\mu}(z) := \int_{\rr} \frac{1}{z-t} \, d\mu(t) \text{.} \]
    This is a holomorphic map from the upper half-plane to the lower half-plane. Among other things, the Cauchy transform encodes weak convergence: if $(\mu_n)_{n \geq 1}$ is a sequence in $\MM(\rr)$ and $\mu \in \MM(\rr)$, then $\mu_n \to \mu$ weakly if and only if $G_{\mu_n}(z) \to G_{\mu}(z)$ pointwise, see, for instance \cite[Remark 3.12, Theorem 3.13]{mingo2017free}.
\end{notation}

%%%%%%%%%%%%%%%%%%%%%%%%%%%%%%%%%%%%%%%%%%%%%%%%%%%%%%%%%%%%%%%%%%%%%%%%%%
\subsection{Fractional free convolution powers and \texorpdfstring{$\freepower{s}$}{Fs} }

The free convolution $ \mu\boxplus \nu $ of two distributions $\mu,\nu$ on $\rr$ can be computed analytically using the $R$-transform of $\mu$, defined as
$$ R_\mu(z) :=G_\mu^{\langle -1\rangle}(z)-\frac{1}{z}, $$
where $ G_\mu^{\langle -1\rangle} $ is the inverse function of the Cauchy transform $G_\mu$. The free convolution $\mu\boxplus \nu$ is then the unique distribution on $\rr$ such that
$$ R_{ \mu\boxplus \nu }=R_\mu + R_\nu. $$

For each integer $n\geq 1$,  the distribution $ \mu^{\boxplus n}:= \underbrace{\mu\boxplus \cdots \boxplus \mu}_{n} $ is called the free (additive) convolution power of $ \mu$. It is shown in \cite{nica1996multiplicative} and \cite{bercovici1995superconvergence} that the notion of free convolution power can be extended to any real number $ 
t\geq 1 $, such a distribution $ \mu^{\boxplus t} $ is usually called the fractional free convolution power of $\mu$, and it is determined by the relation
$$ R_{ \mu^{\boxplus t} }: = tR_{\mu} .$$

In this work, the dilated free additive convolution power will play a fundamental role.

\begin{notation}
Given $s\geq 1$, and $\mu\in \MM(\rr)$  we denote the dilation of the free convolution power by 
$$\freepower{s}(\mu):= \dil{1/s}\mu^{\boxplus s}.$$
\end{notation}
Notice that $\freepower{s}: \MM(\rr) \to \MM(\rr)$ and $\{\freepower{s}\}_{s\geq 1}$ again form a semigroup under convolution, namely:
$$\freepower{s} \circ \freepower{t} = \freepower{st} \qquad \text{for } t,s \geq 1.$$

We now restate with our notation, the result advertised in the introduction, on the fact that when looking at the root distribution of polynomials, differentiation corresponds to the operation of taking free convolution powers in free probability \cite{hoskins2023dynamics}, see also \cite{steinerberger2019nonlocal,steinerberger2023free,arizmendi2023finite,arizmendi2024s}. Recall from \eqref{eq:def.meas} that given a polynomial $p$ of degree $n$, we denote by $\meas{p}$ its empirical root distribution, which has an atom of size $\tfrac{1}{n}$ at each root of $p$.

\begin{theorem}\label{thm:convegence.of.iterative.derivatives}
Fix a sequence of polynomials $p_j$ of increasing (precise) degrees $n_j$ such that $\meas{p_j}$ converges weakly to a probability measure $\nu \in \MM(\rr)$. Let $s\geq 1$ be a parameter and consider a sequence $k_j$ such that $\lim_{j\to\infty}\frac{n_j}{k_j}=s$. Then we have the following weak limit
$$\lim_{j\to\infty} \meas{\diff{k_j}{n_j} p_j}=  \freepower{s} \nu.$$ 
\end{theorem}

%%%%%%%%%%%%%%%%%%%%%%%%%%%%%%%%%%%%%%%%%%%%%%%%%%%%%%%%%%%%%%%%%%%%%%%%%%
\section{Polar free convolution powers for measures}\label{sec:polarfreeconvolution}
%%%%%%%%%%%%%%%%%%%%%%%%%%%%%%%%%%%%%%%%%%%%%%%%%%%%%%%%%%%%%%%%%%%%%%%%%%
In this section we transfer our study of polar differentiation to the analogous operation on Borel probability measures on $\rr$. The intuition comes from Theorem \ref{thm:convegence.of.iterative.derivatives} that asserts that differentiation tends to fractional free convolution. Since polar derivatives can be obtained as a conjugate of usual differentiation with a Möbius transform, in the limit we should obtain free fractional convolution after a conjugation by a Möbius transform. This type of operation has not been studied before.

\subsection{Extension to infinity and definition}
The notion of polar free convolution power is better understood when working with probability measures on 
$\hat{\rr}:=\rr\cup\{\infty\}$, so we first extend the standard notion of free convolution power to this setting. 

A natural extension comes from the correspondence in Theorem \ref{thm:convegence.of.iterative.derivatives}: If $ p\in \pol_n(\cc) $ has precise degree $n-k$, then its root distribution is $$ \meas{p} = \frac{k}{n}\delta_\infty + \frac{n-k}{n}\meas{q} $$
where $q\in \pol_{n-k}(\cc)$ is the same polynomial as $p$ but with formal degree $ n-k $. After differentiating $(n-m)$-times, we obtain that
\[  \meas{\diff{m}{n}p} =
    \frac{k}{m}\delta_\infty + \frac{m-k}{m}\meas{\diff{m-k}{n-k}q},\quad \quad k\leq m<n, \]
while $ \diff{m}{n}p = 0 $ when $ m<k $.  Then one can devise an extended notion of free convolution powers by letting $ \frac{n}{m}\to t $ and $ \frac{k}{n} \to s $. 

To be precise, recall that for $s\geq 1$ and $\nu\in \MM(\rr)$ we use the notation
$\freepower{s}(\nu):= \dil{1/s}\nu^{\boxplus s}$. These definition can be extended to a measure $\mu \in \MM(\hat{\rr}) $ as follows.

\begin{definition}\label{def: extended free convolution power}
    If $ \mu = s\delta_\infty + (1-s)\nu \in \MM(\hat{\rr}) $ with $ \nu \in \MM(\rr) $ and $ 0\leq s\leq 1 $, then $ \freepower{t}(\mu) $ is defined to be
\begin{equation}
\label{eq:freepower.gen.infty}
    \freepower{t}(\mu) = \begin{cases}
        ts\delta_\infty + (1-ts)\freepower{\frac{t-ts}{1-ts}}\nu,& ts\leq1,\\
        \delta_\infty, & ts> 1.
    \end{cases}
\end{equation}
\end{definition}

\begin{remark}
 In Theorem \ref{thm:generalized.convegence.of.iterative.derivatives}
we will show how one can generalize Theorem \ref{thm:convegence.of.iterative.derivatives} using the new notion of $\freepower{s}$ in $\hat{\rr}$.

Notice also that this definition of $\freepower{s}$ is compatible with the recent development \cite[Proposition 7.1]{jalowy2025zeros} where the free convolution of two measures with atoms at infinity is studied. This is not a surprise, as both developments arise from the connection to differentiation of polynomials.

Finally, we should mention that in the case $ts> 1$ there is no natural way to define $\freepower{t}\mu$, and we might as well leave it undefined. We opted to define $\freepower{t}\mu=\delta_\infty$ because in this case one should formally have $ \freepower{t}\mu = \freepower{ts}\freepower{\tfrac{1}{s}}\mu = \freepower{ts}\delta_\infty $, while one also expects $ \freepower{ts} $ to preserve point measures.
\end{remark}

Once we have a notion of free convolution powers in $\hat{\rr}$, we can extend it to polar free convolution powers using Möbius transforms. Recall from Notation \ref{nota:mobius.measure} that given $ \mu \in \MM(\hat{\rr})$ and $ T\in\mobr$, we denote by $ T_*\mu $ the pushforward measure $ T_*\mu(E) := \mu( T^{-1}(E) )\in \MM(\hat{\rr}) $. 

\begin{definition}
\label{def:poler.free.power}
Given $a\in \hat{\rr}$, let $ T\in\mobr$ be defined as $T(z)=\frac{1}{z-a}$ if $a\in \rr$, and define $T(z)=z$ if $a=\infty$. For $s\geq 1$, we define the polar free power of the probability measure $\mu\in\MM(\hat{\rr})$ as the measure 
$$ \pfreepower{a}{s} \mu := T^{-1}_*\freepower{s} T_*\mu \in\MM(\hat{\rr}).$$
\end{definition}
Notice that the particular case $a=\infty$ corresponds to the usual fractional convolution $\pfreepower{\infty}{s} = \freepower{s}$ for $s\geq 1$. Also, it is readily seen from the definition, that the maps $\{ \pfreepower{a}{s}\}_{s\geq 1}$ form a semigroup under convolution:
\begin{equation}
\label{eq:semigroup.polar.free.power}
 \pfreepower{a}{s}\circ \pfreepower{a}{t}=\pfreepower{a}{s+t} \qquad \text{for } s,t\geq 1.   
\end{equation}

\begin{remark}\label{rmk:defnofpfreepower}
Notice that the mass at $a$ under $\pfreepower{a}{t}$ behaves in the same way the mass at $\infty$ behaves under $\freepower{t}$. Indeed, if we write $\mu$ as the sum $ \mu = s\delta_{a}+(1-s)\nu\in \MM(\hat{\rr}) $ with $ \nu \in \MM(\hat{\rr}/\{a\}) $, then
    \begin{equation}
\label{eq:pfreepower.gen.}
    \pfreepower{a}{t}(\mu) = \begin{cases}
        ts\delta_a + (1-ts)\pfreepower{a}{\frac{t-ts}{1-ts}}\nu,\quad ts<1,\\
        \delta_a, \quad ts\geq 1,
    \end{cases}
\end{equation}
where we have $ \pfreepower{a}{\frac{t-ts}{1-ts}}\nu \in \MM(\hat{\rr}/\{a\})$. In particular, we also have $ (\pfreepower{a}{t} \mu)(\{a\}) = \min\{1,t\mu(\{a\})\} $.
\end{remark}

\subsection{Definition using arbitrary Möbius transform.}
In this section we prove that one can obtain $\pfreepower{a}{s}$ by conjugating $\freepower{s}$ with an arbitrary Möbius transform $T$ satisfying $T(a)=\infty$. We first notice that the free fractional powers commute with Möbius transforms that preserve $\infty$.

\begin{lemma}
\label{lem:freepower.commutes.mob.infty}
Let $t\geq 1$ and let $T\in\mobr$ with $T(\infty)=\infty$, then $\freepower{t}\circ T_* = T_*\circ\freepower{t}$ on $\MM(\hat{\rr})$.
\end{lemma}

\begin{proof}
By Remark \ref{rem:Mobius.fix.infty} we know that if $T$ preserves $\infty$, then it can be seen as composition of a dilation with a shift. So it is enough to check that $\freepower{t}$ commutes with dilations and shifts.

First, we prove this for $\mu\in \MM(\rr)$. To check that $\freepower{t}$ commutes with a dilation, we use the distributivity of dilation over the additive convolution. Indeed, for $a\in\rr$ and $\mu\in\MM(\rr)$ one has
$$\freepower{t}(\dil{a}\mu)=\dil{1/t}(\dil{a}\mu)^{\boxplus t}=\dil{1/t}\dil{a}(\mu)^{\boxplus t}
=\dil{a}\dil{1/t}\mu^{\boxplus t}=\dil{a}(\freepower{t}(\mu)).$$

To check that $\freepower{t}$ commutes with a shift we use that $\shift{b}(\mu)=\mu\boxplus \delta_b$ for $b\in\rr$. A simple computation yields
$$ \freepower{t} (\shift{b}\mu) = \dil{1/t}( \mu\boxplus \delta_b )^{\boxplus t} = \dil{1/t}( \mu^{\boxplus t}\boxplus \delta_{b}^{\boxplus t} )= (\dil{1/t}\mu^{\boxplus_t})\boxplus (\dil{1/t}\delta_{bt}) = \shift{b}(\freepower{t} (\mu)).$$

For a general $ \mu \in \MM(\hat{\rr}) $, we write $\mu = s\delta_\infty + (1-s)\nu$ with $ \nu \in \MM(\rr) $. Since shifting does not affect the atom at $\infty$ we can use Definition \ref{def: extended free convolution power},  to obtain $ \freepower{t} (\shift{b}\mu)=\shift{b}(\freepower{t} (\mu)) $. Finally,
$$ \freepower{t}(\dil{a}\mu) = ts\delta_\infty + (1-ts)\freepower{\frac{t-ts}{1-ts}}\dil{a}\nu = ts\delta_\infty + (1-ts)\dil{a}\freepower{\frac{t-ts}{1-ts}}\nu = \dil{a}(\freepower{s}(\mu)),$$
as desired.
\end{proof}

As a consequence of this result, we obtain the following.

\begin{corollary}
\label{cor:free.power.mob.a.infty}
Let $s\geq 1$ and consider $T,S\in\mobr$ with $T(a)=S(a)=\infty$, then 
$$ T^{-1}_*\freepower{s} T_*= S^{-1}_*\freepower{s} S_* \qquad \text{on }\MM(\hat{\rr}).$$
\end{corollary}

\begin{proof}
Consider the transform $TS^{-1}$. By assumption we have that $T(S^{-1}(\infty))=T(a)=\infty$, so it preserves $\infty$. Therefore, Lemma \ref{lem:freepower.commutes.mob.infty} implies that 
$$ \freepower{s} T_*S^{-1}_*= T_*S^{-1}_*\freepower{s}.$$ 
The conclusion follows from composing with $T^{-1}_*$ on the left and $S_*$ on the right.
\end{proof}

\begin{remark}
In particular, letting $T$ in Corollary \ref{cor:free.power.mob.a.infty} be the map in Definition \ref{def:poler.free.power}, we have the alternative definition of polar free power as the map
$$ \pfreepower{a}{s} \mu := S^{-1}_*\freepower{s} S_*\mu,$$
where $S\in\mobr$ with $S(a)=\infty$.    
\end{remark}

To finish this section, we show an analogue of Proposition \ref{prop:FintertwinedbyMobius} for measures. 

\begin{lemma}\label{lem:FintertwinedbyMobius}
If $a,b\in \hat{\rr}$ and $T\in\mobr$ is such that $T(a)=b$, then $T_* \circ \pfreepower{a}{t}= \pfreepower{b}{t} \circ T_* $ on $\MM(\hat{\rr}) $ for all $t\geq 1$.
\end{lemma}

\begin{proof}
    It suffices to show $ T_*^{-1} \pfreepower{b}{t}T_* =\pfreepower{a}{t} $. Let $S\in \MM(\hat{\rr})$ such that $ S( b ) = \infty $, then we have $$ T_*^{-1} \pfreepower{b}{t}T_* = (ST)_*^{-1} \freepower{t}(ST)_*. $$ Since $ S(T(a)) = \infty $, we obtain $ (ST)_*^{-1} \freepower{t}(ST)_* =  \pfreepower{a}{t}$.
\end{proof}

\subsection{Asymptotic root distribution under repeated polar differentiation}

Here we notice that with the new notion of $\freepower{s}$ in $\hat{\rr}$, one can generalize Theorem \ref{thm:convegence.of.iterative.derivatives} to $\MM(\hat{\rr})$.
\begin{theorem}\label{thm:generalized.convegence.of.iterative.derivatives}
    Fix a sequence of polynomials $p_j$ of increasing formal degrees $n_j$ such that $\meas{p_j}$ converges weakly to a probability measure $\mu \in \MM(\hat{\rr})$. Let $t\geq 1$ be a parameter such that $ t\mu(\{\infty\})\leq 1 $ and consider a sequence $m_j$ such that $\lim_{j\to\infty}\frac{n_j}{m_j}=t$ and $ \diff{m_j}{n_j} p_j\neq 0 $. Then
we have the weak limit
$$\lim_{j\to\infty} \meas{\diff{m_j}{n_j} p_j}=  \freepower{s}\mu.$$
\end{theorem}
\begin{proof}
Our approach is to reduce it to the case for $\MM({\rr})$ and then invoke Theorem \ref{thm:convegence.of.iterative.derivatives}. To do this, suppose that $p_j$ has precise degree $n_j-k_j$. Since $\meas{p_j}(\{\infty\}) = \frac{k_j}{n_j}$ and $ \meas{p_j}\to \mu $, we have that $ \frac{k_j}{n_j}$ converges to $s:= \mu(\{\infty\}) $. In particular, if we let $q_j$ be the same polynomial as $p_j$ but with formal degree $n_j-k_j$, then 
    $$ \meas{q_j} \to \frac{1}{1-s} \mu|_{ \rr }=:\nu. $$
    Note that we have $ m_j\geq k_j $ as $  \diff{m_j}{n_j} p_j \neq 0$. In particular, we also have $\lim \frac{m_j}{n_j} = \frac{1}{t}\geq \lim \frac{k_j}{n_j} = s  $. Theorem \ref{thm:convegence.of.iterative.derivatives} yields
    $$ \lim_{j\to \infty }\meas{ \diff{m_j-k_j}{n_j-k_j} q_j} = \freepower{\frac{t-ts}{1-ts}}\nu. $$
    Therefore, we conclude
    \begin{align*}
        \lim_{j\to\infty} \meas{\diff{m_j}{n_j} p_j} &= \lim_{j\to\infty} \left( \frac{ k_j}{m_j}\delta_\infty + \frac{m_j-k_j}{m_j} \lim_{j\to \infty }\meas{ \diff{m_j-k_j}{n_j-k_j} q_j}  \right)\\
        &= ts\delta_\infty + (1-ts)\freepower{\frac{t-ts}{1-ts}}\nu,
    \end{align*} 
as desired.
\end{proof}

We are now ready to prove  
Theorem \ref{thm:asymptotic.polar.diff.intro} advertised in the introduction. Below we will restate in a slightly more general form where we allow measures in $\hat{\rr}$.

\begin{theorem}
\label{thm:asymptotic.polar.diff}
Fix a sequence of polynomials $p_j$ of increasing formal degrees $n_j$ such that $\meas{p_j}$ converges weakly to a probability measure $\mu \in \MM(\hat{\rr})$. Let $t\geq 1$ be a parameter such that $t \mu(\{a\})\leq 1$ and consider a sequence $m_j$ such that $\lim_{j\to\infty}\frac{n_j}{m_j}=t$ and $ \polars{a}{m_j}{n_j} p_j\neq 0 $. Then the following weak limit holds:
$$\lim_{j\to\infty} \meas{\polars{a}{m_j}{n_j} p_j}=  \pfreepower{a}{t}\mu.$$
\end{theorem}

\begin{proof}
Let $T\in\mobr $ be a Möbius transform mapping $ a $ to $\infty$. Since $ T_* $ is weakly continuous, see Remark \ref{rem:Mob.cont}, then $ \meas{T_* {p}}=T_*\meas{p} \to T_*\mu$. By Theorem \ref{thm:generalized.convegence.of.iterative.derivatives}, we conclude that

\begin{align*}
    &\lim_{j\to\infty} \meas{\polars{a}{m_j}{n_j} p_j}=\lim_{j\to \infty} \meas{T_*^{-1}\diff{m_j}{n_j}T_* p} = \lim_{j\to \infty} T_*^{-1}\meas{\diff{m_j}{n_j}T_* p}\\
    &= T_*^{-1}\lim_{j\to \infty} \meas{\diff{m_j}{n_j}T_* {p}} = T_*^{-1}\freepower{t}T_*\mu= \pfreepower{a}{t}\mu,
\end{align*}
as desired.
\end{proof}

\subsection{Preservation of stochastic order}
We now show that our new operation behaves well with respect to the stochastic order on measures that keeps track of the relative position of the mass.

\begin{notation}[Stochastic order on measures]
Given $\mu, \nu\in \MM(\rr)$ we say that $\mu$ is dominated by $\nu$, and denote it $\mu \ll \nu$, if their cumulative distribution functions satisfy
\[\mu((-\infty,c]) \geq \nu((-\infty,c]) \qquad \text{ for all }c\in \rr.\]
\end{notation}

In \cite[Propositions 4.15 and 4.16]{bercovici1993free} it was shown that the free additive and multiplicative convolutions preserve $\ll$. These results can be readily extended to free fractional convolution (see for instance \cite[Corollary 4.2]{arizmendi2024s}):
\begin{equation}
\label{eq:partial.order.free.fractional}
\mu,\nu\in \MM(\rr) \text{ with }\mu \ll \nu \qquad \Rightarrow \qquad  \freepower{t}\mu \ll \freepower{t}\nu,\quad \text{for all } t\geq 1.    
\end{equation}

A similar result holds for the polar free powers.
\begin{proposition}
    If $ a\in \rr$, $ \mu,\nu\in \MM((a,\infty))$, and $ \mu\ll \nu$, then
    $$ \pfreepower{a}{t}\mu \ll \pfreepower{a}{t}\nu,\quad \forall t\geq 1. $$
    The same is true when $\mu,\nu\in \MM((-\infty,a))$.
\end{proposition}
\begin{proof}
Follows from conjugating \eqref{eq:partial.order.free.fractional} by the M\"{o}bius transform $ T(z)=\tfrac{1}{a-z} $ and noticing that $T$ is strictly increasing on $ (a,\infty) $. For the second claim just observe that $T$ is strictly decreasing on $ (-\infty,a) $.
\end{proof}

Moreover, if we fix $\nu$ and change $a\in \rr$, the partial order $\ll$ is still preserved.
\begin{proposition}
    For $a<b$ and $ \nu\in \MM((b,\infty)) $, we have
    $$ \pfreepower{a}{t}\nu \ll \pfreepower{b}{t}\mu,\quad \forall t\geq 1. $$
    Similarly, if $  \mu\in \MM((-\infty,a)) $, then $ \pfreepower{a}{t}\mu \ll \pfreepower{b}{t}\mu $.
\end{proposition}
\begin{proof}
Let $\varepsilon>0$ so that $ \text{supp}(\mu)\subset [b+\varepsilon,\infty) $. Choose a sequence of real rooted polynomials $ p_j $ with (precise) degree $ n_j$ such that $ \meas{p_n} \to \mu $ weakly and $ \lambda_{i}(p_n)\geq b+\varepsilon $ for all $i$ and $n$. By Theorem \ref{thm:asymptotic.polar.diff}, if $ (m_j)_{j\geq 1} $ is a sequence with $\lim_{j\to \infty} \tfrac{n_j}{k_j}=t$, then $$ \lim_{j\to\infty} \meas{\polars{a}{m_j}{n_j} p_j}=  \pfreepower{a}{t}\mu \quad \text{and} \quad  \lim_{j\to\infty} \meas{\polars{b}{m_j}{n_j} p_j}=  \pfreepower{b}{t}\mu .$$ 
Recall from Lemma \ref{lem:polar.derivatives.partial.order}  that $ \meas{\polars{a}{m_j}{n_j} p_j}\ll \meas{\polars{b}{m_j}{n_j} p_j} $. Taking the limit on both sides we conclude that $ \pfreepower{a}{t}\mu \ll \pfreepower{b}{t} \mu$.
\end{proof}

\subsection{Atoms under the action of \texorpdfstring{$\pfreepower{a}{s}$}{Fsa} }

The study of the atoms of $\pfreepower{a}{s}$ can be reduced to the case of the fractional free convolution $\freepower{s}$ using the conjugation via Möbius transform. Recall that Belinschi and Bercovici \cite[Theorem 3.1]{belinschi2004atoms} showed that given $\mu\in \MM(\rr)$, 
then 
\begin{equation}
\label{eq:BB04.atoms}
 b\in \rr\text{ is an atom of }\freepower{s}\mu \quad \Leftrightarrow \quad b \text{ is an atom of $\mu$ with }\mu(\{b\})\geq 1-\frac{1}{s}.   
\end{equation}
 In this case one has that 
$$ \freepower{s}\mu(\{b\}) = \max\{0,1-s(1-\mu(\{b\}))\} .$$

We now present the analogue result for the polar free map.

\begin{proposition}\label{prop:sizeofatom}
    Let $\mu\in\MM({\hat{\rr}})$, and $a,b\in \hat{\rr}$ with $a\neq b$. If $ 0\leq \mu(\{a\})<\tfrac{1}{s} $, then \begin{equation}
\label{eq:atoms.pfreepower}
 b\in \rr\text{ is an atom of }\pfreepower{a}{s}\mu \quad \Leftrightarrow \quad b \text{ is an atom of $\mu$ with }\mu(\{b\})\geq 1-\tfrac{1}{s}.   
\end{equation}
Moreover, in this case we get
    $$ \pfreepower{a}{s}\mu(\{b\}) = \max\{0,1-s(1-\mu(\{b\}))\}. $$
\end{proposition}
\begin{proof}
We first note that by conjugating $\mu$ with a map $T\in \mobr$ satisfying $T(a)=\infty$, we may assume without loss of generality that $a=\infty$. So let us assume this and denote $w:=\mu(\{\infty\})$ to ease notation.

If $w=0$, then the claim is equivalent to \cite[Theorem 3.1]{belinschi2004atoms}. Then we are just left to check the case when $a=\infty$ and $ 0< w<\tfrac{1}{s} $. Decomposing $\mu$ as $\mu = w\delta_{\infty} + (1-w)\nu$, by  Definition \ref{def: extended free convolution power} one has that
$$ \freepower{s}\mu =  sw\delta_\infty + (1-sw)\freepower{\frac{s-sw}{1-sw}}\nu.$$
    Then $\nu(\{b \})=\frac{\mu(\{b\})}{1-w}$, and \eqref{eq:BB04.atoms} yields that $ \freepower{s}\mu $ has an atom at $b$ if and only if $\mu$ has an atom at $b$ such that
    $$ \nu(\{b\})=\frac{\mu(\{b\})}{1-w}\geq 1- \frac{1-sw}{s-sw}.$$
    After simplification the latter is equivalent to $ \mu(\{b\})\geq 1-\frac{1}{s} $. 
    
    Finally, if the above is true, then
    \begin{align*}
         \freepower{s}\mu(\{b\}) &= (1-sw)\max\left\{ 0,1- \frac{s-sw}{1-sw}\left(1-\frac{\mu(\{b\})}{1-w} \right) \right\}\\
         &= \max\left\{ 0,1-s(1-\mu(\{b\})) \right\},
    \end{align*}
    as desired.
\end{proof}

\begin{corollary}\label{cor:polarpowerpreservesatomlocation}
    For arbitrary $ a,b_1,\cdots,b_k\in \hat{\rr} $, the set of measures $\MM( \hat{\rr}/\{b_1,\cdots,b_k\} ) $ that vanish at $\{b_1,\cdots,b_k\}$ (see notation in \eqref{eq:not.measure.vanishes}) is invariant under $ \pfreepower{a}{s} $:
    $$ \pfreepower{a}{s}\MM( \hat{\rr}/\{b_1,\cdots,b_k\} )\subset \MM( \hat{\rr}/\{b_1,\cdots,b_k\} ),\quad \forall s\geq 1. $$
\end{corollary}

\section{Connection of \texorpdfstring{$\pfreepower{0}{t}$}{F0t} to multiplicative free convolution.}
\label{sec:connection.to.free.mult}

Interestingly enough the map $\pfreepower{0}{t}$ is somewhat related to multiplicative free convolution of measures. First we will stablish this relation at the level of polynomials, in this framework the connection is more intuitive.

Recall from \cite{marcus2016polynomial} that the finite free multiplicative convolution of two polynomials
$$ p(x)=\sum_{k=0}^n x^{n-k}(-1)^k \binom{n}{k} \coef{k}{p}\qquad \text{and} \qquad q(x)=\sum_{k=0}^n x^{n-k}(-1)^k \binom{n}{k} \coef{k}{q},$$
of degree at most $n$ is the polynomial
$$[p\boxtimes_n q](x)=\sum_{k=0}^n x^{n-k}(-1)^k \binom{n}{k} \coef{k}{p}\coef{k}{q}.$$

The 0-polar derivative is used in \cite[Lemma 4.9]{marcus2016polynomial} to provide a degree reduction formula for $\boxtimes_n$. In practice, this means that the 0-polar derivative (at degree $n$) can be understood as finite free multiplicative convolution with the polynomial $Q_n=n(x-1)^{n-1}$. Namely, for every polynomial $p$ it holds that
$$\polar{0}p=p\boxtimes_n Q_n.$$
This means that 0-polar differentiation of $p$ can be achieved by simply multiplying $p$ by a polynomial of degree $n-1$. In general, what the operation $\polars{0}{k}{n}$ is doing to $p$, is a multiplicative convolution by
$$Q_{n,k}:= \falling{n}{k} (x-1)^k,$$
where $\falling{n}{k}=n(n-1) \cdots(n-k+1)$. Notice that this operation corresponds to forgetting the first $n-k$ coefficients of $p$, while adjusting the rest by a factor. More specifically, a direct consequence of \cite[Lemma 4.9]{marcus2016polynomial} is the following.

\begin{lemma}
Let $Q_{n,k}:= \falling{n}{k} (x-1)^k$, then
$$\polars{0}{k}{n}p=p\boxtimes_n Q_{n,k}.$$    
\end{lemma}

To look at the analogous result for measures, we can use the multiplicative convolution with a measure supported on the extended real line, which was recently studied by Jalowy, Kabluchko and Marynych \cite[Proposition 6.6]{jalowy2025zeros}.
Notice that if we consider $Q_{n,k}$ as a polynomial of formal degree $n$ (with precise degree $k$) then we may interpret $\meas{Q_{n,k}}=\frac{k}{n}\delta_1+\frac{n-k}{n}\delta_\infty$. Thus, if we let $n\to\infty$ while $\frac{n}{k}\to s$ we obtain that
\begin{equation}
\label{eq:formula.pfreepower0}
\frac{1}{s}\pfreepower{0}{s} \nu + \frac{s-1}{s}\delta_{\infty}=  \nu \boxtimes \left(\frac{1}{s}\delta_1+\frac{s-1}{s}\delta_\infty\right).   
\end{equation}

\begin{remark}
Notice that \eqref{eq:formula.pfreepower0} is precisely the reciprocal version of the following well-known relation in free probability (see for instance \cite[Exercise 14.21]{nica2006lectures}):
$$ \frac{1}{s}\freepower{s} \nu + \frac{s-1}{s}\delta_{0} = \nu \boxtimes \left(\frac{1}{s}\delta_1+\frac{s-1}{s}\delta_0\right).$$    
This is not a surprise, as taking multiplicative inverses preserves freeness.
\end{remark}

Observe that \eqref{eq:formula.pfreepower0} coincides with the fact that $\pfreepower{0}{s}\nu$ is the reciprocal of a compression of the reciprocal measure of $\nu$. Indeed, this can be seen from the fact that $\pfreepower{0}{s}:= T_* \freepower{s} T_*$, where $T(z):=z$. Since the $S$-transform behaves well under taking inverses and doing compressions, one can compute the $S$-transform of $\pfreepower{0}{s}\nu$ using the corresponding formulas. Alternatively, taking advantage of the connection to polynomials, one can look at the coefficients of $Q_{n,k}$.

\begin{lemma}
The $S$-transform of the measure in \eqref{eq:formula.pfreepower0} is equal to
$$  \frac{t+1-s}{t}S_\nu(t) \qquad \text{for }t\in(-1,s-1).$$
\end{lemma}

\begin{proof}
Notice that
$$Q_{n,k}= \sum_{j=0}^k \binom{k}{j} (-1)^{k-j} x^{j}.$$
Thus, for $t \in(0,s)$ we have that in the limit $n\to\infty,\, \tfrac{k}{n}\to s,\, \tfrac{j}{n}\to t \in(0,s)$  the logarithm of the coefficients of the polynomial tends to
$$\frac{1}{n}\log \binom{k}{j} \to -t \log t -(s-t)\log (s-t) +s\log s.$$
Therefore, the exponential profile is 
$$g(t)=-t \log t -(s-t)\log (s-t) +s\log s \qquad 
\text{for }  t\in(0,s),$$
and using \cite[Equation (65)]{jalowy2025zeros} we obtain that the $S$-transform is 
$$S_{\frac{1}{s}\delta_1+\frac{s-1}{s}\delta_\infty} (t)=-\frac{t+1}{t}e^{q'(t+1)}= \frac{t+1-s}{t} \qquad \text{for }t\in(-1,s-1) .$$
The conclusion follows from the multiplicative property of the $S$-transform in the extended real line \cite[Proposition 6.6]{jalowy2025zeros}.
\end{proof}

Let us now take the opportunity to briefly mention how we can use the connection between 0-polar derivatives and multiplicative convolution to realize $0$-polar derivatives in terms of random matrices.

\begin{remark}
If we let $p$ be a polynomial of degree $n$ with all positive roots. Denote by $I$ the $n\times n$ identity matrix and by $P$ the diagonal matrix with $P_{ii}=1$ for $i=1,\dots, n-1$ and $P_{nn}=0$. Consider self-adjoint $n\times n$ matrix $A$ such that $\det[xI-A]=p(x)$, then it is not hard to check that
$$ x^n [\polar{0} p](\tfrac{1}{x}) = \ee \det [xI-PUA^{-1}U^*],$$
where the expectation is taken over Haar unitary matrices $U$. Alternatively, one could also integrate of Haar orthogonal matrices or permutation matrices.

Notice that in the the left-hand side we are artificially applying the inverse formula to the polynomial. It would be interesting to find a way to realize this type of operation directly using the expected characteristic polynomial of a random matrix.
\end{remark}

To conclude this section, let us mention that since $a$-polar derivative can be obtained from $0$-polar derivative after a shift, one could interpret doing polar derivatives in terms of finite free multiplicative and additive convolution. In the limit we have $\pfreepower{a}{s}= \shift{a} \pfreepower{0}{s} \shift{-a}$, which in turn can be expressed in terms of additive convolutions with Dirac delta measures and multiplicative convolution using \eqref{eq:formula.pfreepower0}. This provides some insight of the operations in free probability. However, since multiplicative and additive convolutions are computed using different transforms, it is not clear how to use this approach to study the behavior of measures after applying the polar power maps more than once, for instance, in the commutation relation in Theorem \ref{thm:commutation.intro}.

%%%%%%%%%%%%%%%%%%%%%%%%%%%%%%%%%%%%%%%%%%%%%%%%%%%%%%%%%%%%%%%%%%%%%%%%%%
\section{Proof of the commutation relation}\label{sec:pfofcommutationrelation}
%%%%%%%%%%%%%%%%%%%%%%%%%%%%%%%%%%%%%%%%%%%%%%%%%%%%%%%%%%%%%%%%%%%%%%%%%%

In this section we provide two proofs of Theorem \ref{thm:commutation.intro}, which is our second main result. The first proof relies on polar differentiation and Theorem \ref{thm:asymptotic.polar.diff.intro}, this approach actually yields a more general result where one can include measures on the extended real line. Our second approach uses an operator model to directly prove the result for the real line. It is also natural to consider a more analytical approach using partial differential equations. Although this approach yields some insightful results, we were not able to complete a rigorous proof with this method, we refer the reader to Section \ref{sec:analytical}, for the details on this approach.

%%%%%%%%%%%%%%%%%%%%%%%%%%%%%%%%%%%%%%%%%%%%%%%%%%%%%%%%%%%%%%%%%%%%%%%%%%
\subsection{Proof of Theorem \ref{thm:commutation.intro} using polar differentiation}

We begin by stating the slightly more general version of Theorem \ref{thm:commutation.intro}.

\begin{theorem}\label{thm:commutation}
Let $ s,t\geq 1$, and $a,b\in \hat{\rr} = \rr\cup \{\infty\}$. Let also $s',t' >1$ be constants determined by the relations $ st=s't' $, $ s+s' = 1+st $. Then
$$\pfreepower{a}{s}\pfreepower{b}{t} \mu =\pfreepower{b}{s'}\pfreepower{a}{t'}\mu, \qquad \forall \mu \in \MM(\hat{\rr}).$$
\end{theorem}

\begin{remark}
 Notice that Theorem \ref{thm:commutation.intro} follows directly from Theorem \ref{thm:commutation} by noticing that $ \pfreepower{a}{s} $ and $ \pfreepower{b}{t} $ preserve $ \MM({\rr}/\{a,b\}) $, see Corollary \ref{cor:polarpowerpreservesatomlocation}.
\end{remark}

\begin{proof}[Proof of Theorem  \ref{thm:commutation}]
When $ a=b $, the statement holds trivially as $ st=s't' $, so we may assume $a\neq b$. We now separate in 3 cases.

\textbf{Case 1.} Assume that $t\mu(\{b\})<1$, and $t'\mu(\{a\})< 1 $.  Fix a sequence of polynomials $p_j$ of increasing formal degrees $n_j$ converging to $\mu$. Consider two sequences $k_j,l_j$ with $1\leq l_j < k_j\leq n_j$ such that $ \lim_{j\to\infty} \frac{n_j}{k_j}=t $ and $ \lim_{j\to\infty} \frac{k_j}{l_j} =s$ and $  \polars{a}{l_j}{k_j}\polars{b}{k_j}{n_j} p_j \neq 0$. (Indeed, such sequences exist due to the assumption $t\mu(\{b\}), t'\mu(\{a\})\leq 1 $.) Let $T_a$ be a Möbius map sending $a$ to $\infty$ and $T_b$ be a Möbius map sending $b$ to $\infty$. Applying Theorem \ref{thm:asymptotic.polar.diff} twice, we get $$ \lim_{j\to \infty}\meas{\polars{a}{l_j}{k_j}\polars{b}{k_j}{n_j} p_j}\to \pfreepower{a}{s}\pfreepower{b}{t}\mu. $$ 
On the other hand, since $\polar{a}\polar{b}=\polar{b}\polar{a}$, 
$$\polars{a}{l_j}{k_j}\polars{b}{k_j}{n_j}= \polar{a}^{k_j-l_j}\polar{b}^{n_j-k_j } = \polar{b}^{n_j-k_j }\polar{a}^{k_j-l_j} = \polars{b}{l_j}{n_j-k_j+l_j}\polars{a}{n_j-k_j+l_j}{n_j}.$$
Since $ \lim_{j\to \infty} \frac{n_j-k_j+l_j}{l_j} = st-s+1=s' $ and $ \lim_{j\to \infty}\frac{ n_j}{n_j-k_j+l_j} = \frac{st}{st-s+1} =t' $, applying Theorem \ref{thm:asymptotic.polar.diff} twice again, we obtain 
$$ \lim_{j\to \infty} \meas{\polars{a}{l_j}{k_j}\polars{b}{k_j}{n_j}p_j} = \lim_{j\to \infty}  \meas{\polars{b}{l_j}{n_j-k_j+l_j}\polars{a}{n_j-k_j+l_j}{n_j}p_j} = \pfreepower{b}{s'}\pfreepower{a}{t'}\mu.$$

\textbf{Case 2.} Assume that $ t\mu(\{b\}) \geq 1 $. By Remark \ref{rmk:defnofpfreepower}, we have $ \pfreepower{a}{s}\pfreepower{b}{t}\mu = \pfreepower{a}{s}\delta_b = \delta_b $. On the other hand, by Proposition \ref{prop:sizeofatom} we have $$ (\pfreepower{a}{t'}\mu)(\{b\}) = 1-t'( 1-\mu(\{b\}) )= \frac{s(t\mu(\{b\})-1)+1}{1+st-s},$$
and therefore $\pfreepower{b}{s'}\pfreepower{a}{t'}\mu = \delta_b$ as $ s'(\pfreepower{a}{t'}\mu)(\{b\})= s(t\mu(\{b\})-1)+1\geq 1$.

\textbf{Case 3.} Assume that $ t'\mu(\{a\})\geq 1 $. This is analogous to Case 2, we can also compute $ \pfreepower{b}{s'}\pfreepower{a}{t'}\mu = \delta_a= \pfreepower{a}{s}\pfreepower{b}{t}\mu $.
\end{proof}

\begin{remark}
    We note that in the proof above, the two conditions $  t\mu(\{b\}) \geq 1  $ and $ t'\mu(\{a\})\geq 1 $ cannot happen at the same time as otherwise $$ \mu(\{b\})+\mu(\{a\}) \geq \frac{1}{t}+\frac{1}{t'} = \frac{1}{st}+1 >1,$$
    which is impossible.
\end{remark}

%%%%%%%%%%%%%%%%%%%%%%%%%%%%%%%%%%%%%%%%%%%%%%%%%%%%%%%%%%%%%%%%%%%%%%%%%%
\subsection{Direct proof of Theorem \ref{thm:commutation.intro} via operator model}

Now we want to study the operator model for polar free convolution powers.  Consider a tracial $W^*$-probability space $(M,\tau)$, i.e. $ M $ is a von Neumann algebra with a faithful normal trace $\tau$. We will also denote the algebra of unbounded operators affiliated with $ M $ by $ \widetilde{M} $.

Let $ X \in \widetilde{M}$ be an unbounded self-adjoint operator affiliated with $M$. The distribution of $X$ is then defined as $\tau\circ E_X$ where $ E_X $ is the spectrum measure of $X$. Let $p\in M$ be a projection freely independent from $X$ such that $ \tau(p) = 1/t $. It is well known, see \cite[Application 1.11]{nica1996multiplicative}, that $ pXp\in pMp $ has distribution $ \freepower{t}\mu $ in $ (pMp,\tau_p) $ where $\tau_p(pXp)=\frac{\tau(X)}{\tau(p)}$.
We will see that the operator model for $\pfreepower{a}{t}$ is closely related to inverse block matrices, and we will use it to give another proof of Theorem \ref{thm:commutation.intro}.

We will begin by providing an operator model for $\pfreepower{a}{t}$. Notice that we will use the standard notation $p^{\perp}:=1-p$ to denote the orthogonal complement of a projection $p$.

\begin{proposition}\label{operator model}
    Let $X$ be an unbounded self-adjoint operator $ X $ affiliated with $(M,\tau)$, and let $p\in M$ be a projection free from $ X $. Assume that $X$ has distribution $\mu \in \MM(\rr/\{a\})$ for some $a\in {\rr}$, and assume that $\tau(p)=\tfrac{1}{t}$ for some $t\geq 1$. Then, the unbounded operator
    $$ (p(X-a)^{-1}p)^{-1}+a = pXp- pX\left( p^{\perp}(X-a) p^{\perp}\right)^{-1}Xp \quad \in \widetilde{pMp} $$
    has distribution $ \pfreepower{a}{t}\mu $.
\end{proposition}

\begin{proof}
Let $T\in \mobr$ be the Möbius transform $T(z)=\frac{1}{z-a}$, and recall that $T^{-1}(z)=\frac{1}{z}+a$. From the assumption $\mu(\{a\})=0$, we obtain that $T(X)=(X-a)^{-1} $ is a well-defined affiliated operator, and its distribution is precisely $T_*\mu$, so we have that
$$ T^{{-1}}( pT(X)p ) = \left( p(X-a)^{-1}p  \right)^{-1}+a \quad  \in pMp$$ 
has distribution $ \pfreepower{a}{t}\mu $. Notice that we are taking the inverse of $p(X-a)^{-1}p$ inside $ pMp$. Moreover, one can expand the term $ \left( p(X-a)^{-1}p  \right)^{-1} $ using the inverse formula of block matrices.
    
    Recall that if a block matrix $P = \begin{bmatrix}
        A&B\\
        C&D
    \end{bmatrix} $ is invertible and the block $ A $ is also invertible, then $D-CA^{-1}B$ is invertible and the inverse $P^{-1}$ can be written as
    $$ P^{-1}= \begin{bmatrix}
        A^{-1} + A^{-1}B(D-CA^{-1}B)^{-1} CA^{-1} & -A^{-1}B (D-CA^{-1}B)^{-1} \\
         -(D-CA^{-1}B)^{-1}CA^{-1}  &(D-CA^{-1}B)^{-1} 
    \end{bmatrix}, $$
so the (2,2)-corner is given by
$$\begin{bmatrix}
        0&0\\
        0&1
    \end{bmatrix} P^{-1} \begin{bmatrix}
        0&0\\
        0&1
    \end{bmatrix}= \begin{bmatrix}
        0 & 0 \\
        0 &(D-CA^{-1}B)^{-1} 
    \end{bmatrix}. $$
    
    In our case, we can split the $L^2$ space $L^2M$ into the direct sum $ p^{\perp}L^2M\oplus pL^2M $ and consider $X-a$ as an operator on $ p^{\perp}L^2M\oplus pL^2M $. Thus we can express $ X-a $ as a block matrix
    $$ X-a = \begin{bmatrix}
        p^{\perp}(X-a)p^{\perp}& p^{\perp}Xp\\
        pXp^{\perp} & p(X-a)p
    \end{bmatrix} $$
    and therefore  $p(X-a)^{-1}p $ (the (2,2)-corner of $(X-a)^{-1}$) has inverse in $pMp$ given by
    \begin{align*}
\left(p(X-a)^{-1}p \right)^{-1}&= p(X-a)p- pXp^{\perp} \left( p^{\perp}(X-a)p^{\perp}
 \right)^{-1}p^{\perp} Xp \\
 &= p(X-a)p- pX\left( p^{\perp}(X-a)p^{\perp}
 \right)^{-1}Xp
    \end{align*}
 where again the inverse of $ p^{\perp}(X-a)p^{\perp} $ is taken in $p^{\perp}Mp^{\perp} $.

Therefore, an operator model for $ \pfreepower{a}{t}\mu $ is
 $$ T^{{-1}}( pT(X)p ) = (p(X-a)^{-1}p)^{-1}+ap = pXp- pX\left( p^{\perp}(X-a)p^{\perp}\right)^{-1}Xp, $$
and the conclusion follows.
\end{proof}

Once this is settle, we have all the ingredients required to give an operator model proof of our main result. 

\begin{proof}[Second proof of Theorem \ref{thm:commutation.intro}]

We first prove the case where $b=\infty$ and $a\in \rr$, then we show that the general case can be reduced to the first case.

\textbf{Case 1: $a\in \rr$, $b=\infty$.} Fix $s,t> 1$ and recall that $s',t' >1$ are determined by the relations $ st=s't' $, and $ s+s' = 1+st $. Consider two commuting projections $ p,q\in M $ such that $ \tau(p)=\tfrac{1}{t} $, $\tau(q)=\tfrac{1}{t'} $, and $ p^\perp \leq q $. Then, construct the projection $r := q-p^{\perp} = q+p-1=pq$. To prove that $ \pfreepower{a}{s}\freepower{t}\mu = \freepower{s'}\pfreepower{a}{t'}\mu $ we will use Proposition \ref{operator model} to show that the operator 
$$Y:=rXr-rX( q^{\perp} (X-a) q^{\perp} )^{-1}Xr \quad \in \widetilde{rMr}$$
satisfies $Y \sim \pfreepower{a}{s}\freepower{t}\mu$ and $Y\sim \freepower{s'}\pfreepower{a}{t'}\mu$.
\begin{itemize}
    \item $Y \sim \pfreepower{a}{s}\freepower{t}\mu$. On the one hand, we notice that $r= p+q-1$ is freely independent of $X$ in $M$. Since conjugating by $p$ gives the free convolution power of the joint distribution, we obtain that $r = prp$ is free from $ pXp \sim \freepower{t}\mu$. Notice also that the trace of $r$ is
$$ \tau(r)= \tau(q+p-1) = \tfrac{1}{t'}+\tfrac{1}{t}-1 = \tfrac{1+st-s}{st}+\tfrac{1}{t}-1= \tfrac{1}{st},$$
so $\tau_p(r)=\frac{\tau(r)}{\tau(p)}=\frac{1}{s}$. By Proposition \ref{operator model}, we conclude that
$$Y= rpXpr-rpXp( (p-r) (X-a) (p-r) )^{-1}pXpr \, \sim \, \pfreepower{a}{s}\freepower{t}\mu. $$

\item $Y\sim \freepower{s'}\pfreepower{a}{t'}\mu$. On the other hand, notice that 
$$ qXq-qX( q^{\perp} (X-a) q^{\perp} )^{-1}Xq = T^{-1}(qT(X)q) \, \sim \, \pfreepower{a}{t'}.$$ 
Again, since $r$ is free from  $qT(X)q $ in $ (qMq,\tau_q) $, then $r$ is free from $ T^{-1}(qT(X)q) $.  Since $\tau_q(r)=\frac{\tau(r)}{\tau(q)}=\tfrac{t'}{st} =s' $, then Proposition \ref{operator model} yields that
$$Y= rXr-rX( q^{\perp} (X-a) q^{\perp} )^{-1}Xr \, \sim \,\freepower{s'}\pfreepower{a}{t'}\mu .$$ 
\end{itemize}

\textbf{Case 2: $a,b\in \rr$.} We will reduce this to Case 1. Let $T,S\in\mobr$ be such that $T(a)=\infty$ and $S(b)=\infty$, and assume that $c$ is such that $ TS^{-1} (c)=\infty$.  We compute
$$ \pfreepower{a}{s} \pfreepower{b}{t}= (T^{-1} \freepower{s} T )(S^{-1} \freepower{t}S) = S^{-1} (TS^{-1})^{-1} \freepower{s} (TS^{-1})\freepower{t} S= S^{-1} \pfreepower{c}{s}\freepower{t}S. $$
On the other hand,
$$\pfreepower{b}{s'}\pfreepower{a}{t'}=(S^{-1}\freepower{s'}S)(T^{-1} \freepower{t'} T )=S^{-1}\freepower{s'}(TS^{-1})^{-1} \freepower{t'} (TS^{-1})S= S^{-1} \freepower{s'}\pfreepower{c}{t'}S.$$
By Case 1 we know that $\pfreepower{c}{s}\freepower{t}=\freepower{s'}\pfreepower{c}{t'}$, so we conclude that $ \pfreepower{a}{s} \pfreepower{b}{t}= \pfreepower{b}{s'}\pfreepower{a}{t'}$.
\end{proof}

\begin{remark}
    One can also try to use the same argument to prove the more general version, Theorem \ref{thm:commutation}. However, for general $X$, the element $(X-a)$ might not be invertible if $a$ is in the point spectrum $X$. Therefore, we should consider $ (X-a)^{-1} $ as a ``formal operator'' of the form $(X-a)^{-1}= Y_0+\infty\cdot P$ where $Y_0 \in P^\perp MP^\perp$ and $ P $ is the spectral projection $P = E_{\{a\}}(X)$. By keeping track of the projection $P$ under free compression of the operator model, it is possible to recover Theorem \ref{thm:commutation} for general measures. Since these kind of ``formal operator'' are out of the scope of this manuscript, we will not explore the details.
\end{remark}

%%%%%%%%%%%%%%%%%%%%%%%%%%%%%%%%%%%%%%%%%%%%%%%%%%%%%%%%%%%%%%%%%%%%%%%%%%
\section{Infinite divisibility}\label{sec:infdivisibility}
%%%%%%%%%%%%%%%%%%%%%%%%%%%%%%%%%%%%%%%%%%%%%%%%%%%%%%%%%%%%%%%%%%%%%%%%%%
Recall that a distribution $\mu\in \MM(\rr)$ is called freely infinitely divisible (FID) if there exists $ \mu_t \in  \MM(\rr) $ such that $ \freepower{\frac{1}{t}}\mu_t = \mu $ for all $0<t<1$, or equivalently, if $\mu_t:= \freepower{t}\mu$ can be defined for all $t>0$. When $\mu$ is FID, the distribution $ \mu_t $ is unique since we must have $ R_{\mu_t}(z)=R_\mu(z/t) $.

Given $a\in \hat{\rr}$, we can also define a polar version of FID distributions, now for the semigroup generated by $\pfreepower{a}{t}$.
\begin{definition}
For $a\in \hat{\rr}$, a distribution $\mu\in \MM(\hat{\rr})$ is called $a$-free infinitely divisible ($a$-FID), if for all $0<t<1$, there exists a distribution $ \mu_t \in  \MM(\hat{\rr}) $ such that $ \pfreepower{a}{\frac{1}{t}}\mu_t = \mu $.  The set of all $a$-FID distributions will be denoted by $\text{FID}(a,\hat{\rr})$.
\end{definition}

\begin{remark}
Notice that for the case $a=\infty$, this definition corresponds to a extension of the notion of FID distributions, where one allows $\mu$ to be in $\MM(\hat{\rr})$. It is also not hard to check that if $\mu\in\MM(\rr)$ is FID, then the extension of $\mu$ to $\MM(\hat{\rr})$ is in $\text{FID}(\infty,\hat{\rr})$. 
\end{remark}

Similar to what we have seen along this work, the study of a general $a$-FID distributions can be reduced to the study of $\infty$-FID thanks to help of Möbius transforms.

\begin{lemma}
\label{lem:FIDa.reduction.to.FID}
Let $T\in \mobr$ be such that $T(a) =\infty$ and $s > 1$. Then, $ \pfreepower{a}{s}\nu = \mu $ if and only if  $ \freepower{s} T_*\nu= T_*\mu$. In particular, $ \mu \in \MM(\hat{\rr}) $ is $a$-FID if and only if $ T_*\mu $ is $\infty$-FID.
\end{lemma}

\begin{proof}
If $ \pfreepower{a}{s}\nu = \mu $, then a straightforward computation yields
$$ \freepower{s} (T_*\nu)=   T_*\pfreepower{a}{s}T_*^{-1}T_*\nu = T_*\pfreepower{a}{s}\nu = T_*\mu.$$
The converse follows in a similar way. For the second part we notice that $ \mu \in \MM(\hat{\rr}) $ is $a$-FID, if and only if for all $t\in(0,1)$ there exists a distribution $ \mu_t \in  \MM(\hat{\rr}) $ such that $ \pfreepower{a}{\frac{1}{t}}\mu_t = \mu $. By the previous part, this holds if and only if for all $t\in(0,1)$ one has $ \freepower{\frac{1}{t}} (T_*\mu_t)= T_*\mu$. By definition, the latter means that $ T_*\mu $ is $\infty$-FID.
\end{proof}

We now turn our study to the underlying measure $\mu_t$ satisfying that $\pfreepower{a}{\frac{1}{t}} \mu_t=\mu$. First we notice that this measure is unique, unless $\mu=\delta_a$

\begin{remark}
Observe that for all $a\in \hat{\rr}$ we have that $\delta_a \in \text{FID}(a,\hat{\rr})$. Indeed, Definition \ref{def: extended free convolution power} and Remark \ref{rmk:defnofpfreepower} imply that for $t\geq 1$ and $\mu$ with $ \mu(\{a\})\geq \frac{1}{t} $, it holds that $ \pfreepower{a}{t}\mu=\delta_a $. So there are several measures $\mu$ that yield $\delta_a$. 
\end{remark}

The measure $\delta_a$ is the only one in $\text{FID}(a,\hat{\rr})$ where the underlying measure is not unique, as we will see in the next result.

\begin{lemma}
\label{lem:uniqueness.FID}
Let $a\in \hat{\rr}$ and consider $ \mu\in \MM(\hat{\rr}) $ with $\mu\neq \delta_a$. If for $t\in(0,1)$ there exists a measure $\mu_t$ satisfying $ \pfreepower{a}{\tfrac{1}{t}}\mu_t = \mu $, then the measure $\mu_t$ is unique.
\end{lemma}

\begin{proof}
We first check it in the case $a=\infty$. the idea is to check the restriction of the measures to $\rr$ and use the uniqueness there. 

Assume $ \pfreepower{a}{\tfrac{1}{t}}\mu_t = \mu $ and write  
$$ \mu = s\delta_\infty + (1-s)\nu \qquad \text{and} \qquad \mu_t = s'\delta_\infty + (1-s')\nu',$$
for $s,s'\in[0,1]$ and $\nu,\nu'\in \MM(\rr)$. 

By Definition \ref{def: extended free convolution power}, it holds that 
$$\frac{s'}{t}\delta_\infty + \left(1-\frac{s'}{t}\right)\freepower{\frac{1-s'}{t-s'}}\nu' = s\delta_\infty + (1-s)\nu.$$
In particular, we must have $\frac{s'}{t}=s$ and $\freepower{\frac{1-s'}{t-s'}}\nu'= \nu$. 
 Finally, since $\mu\neq \delta_{\infty}$, then $1-s>0$ and $\nu$ is uniquely determined by $\mu$. From the theory of $FID$ on the real line, we also know that $\nu'$ is uniquely determined by $\nu$. Thus, overall, $ \mu_t$ is uniquely determined by $\mu $, as desired. 

The general case $a\in\rr$, follows from the previous case and Lemma \ref{lem:FIDa.reduction.to.FID}.
\end{proof}

\begin{notation}
In view of Lemma \ref{lem:uniqueness.FID}, if $\mu\in \MM(\hat{\rr})/\{\delta_a\}$ and for $t\in(0,1)$ there exists $\mu_t\in \MM(\hat{\rr})$ satisfying $ \pfreepower{a}{\tfrac{1}{t}}\mu_t = \mu $, then we use the notation
$$ \pfreepower{a}{t}\mu:=\mu_t .$$
\end{notation}

\begin{lemma}
    For $ a,b\in \hat{\rr} $ and $a\neq b$, both $ \text{FID}(a,\hat{\rr}) $ and $\text{FID}(a,\hat{\rr})/\{\delta_a\}$ are invariant under $ \pfreepower{b}{t} $ for all $t\geq 1$.
\end{lemma}
\begin{proof}
If $ \mu=\delta_a $, then we always have $ \pfreepower{b}{t}\delta_a=\delta_a \in \text{FID}(a,\hat{\rr})$. So, let us assume that $ \mu\in \text{FID}(a,\hat{\rr})/\{\delta_a\} $. Using Theorem \ref{thm:commutation.intro} then for every $u>1$ it holds that
    $$ \pfreepower{b}{t}\mu = \pfreepower{b}{t}\pfreepower{a}{u}\pfreepower{a}{\frac{1}{u}}\mu = \pfreepower{a}{1+(u-1)t}\pfreepower{b}{\frac{ut}{1+(u-1)t}}\pfreepower{a}{\frac{1}{u}}\mu.$$
    Therefore, we have 
    \begin{equation}
    \label{eq:aux.commutation}
    \pfreepower{a}{\frac{1}{1+(u-1)t}}\pfreepower{b}{t}\mu = \pfreepower{b}{\frac{ut}{1+(u-1)t}}\pfreepower{a}{\frac{1}{u}}\mu,
    \end{equation} which implies that $ \pfreepower{b}{t}\mu\in  \text{FID}(a,\hat{\rr}) $ as we can choose $ \frac{1}{1+(u-1)t} $ to be any number inside $(0,1)$. Also, $ \pfreepower{b}{t}\mu\neq \delta_a $ as $ \pfreepower{b}{t} $ shrinks the size of the atom at $a$.
\end{proof}

As a corollary, we can extend Theorem \ref{thm:commutation} to allow $s$ to take values in $(0,1)$.

\begin{corollary}
    For $a,b\in \hat{\rr}$, $a\neq b$, and $ \mu \in \text{FID}(a,\hat{\rr})/\{\delta_a\}$, we have
     $$ \pfreepower{a}{s}\pfreepower{b}{t} \mu =\pfreepower{b}{s'}\pfreepower{a}{t'}\mu, \qquad \text{for all } t\geq 1,\, s>0, $$
     where again $ s' := 1+st-s $ and $t' := st/s'$.
\end{corollary}
\begin{proof}
The case when $s\geq 1$ follows directly from Theorem \ref{thm:commutation}. For the case $s\in(0,1)$ we let $u:=\frac{1-s+st}{st}$ in equation \eqref{eq:aux.commutation} above, so that $s=\frac{1}{1+(u-1)t}$, $ s' = 1+st-s=\frac{ut}{1+(u-1)t} $ and $t' = \frac{st}{s'}= \frac{1}{u}$. Then we obtain that 
$$ \pfreepower{a}{s}\pfreepower{b}{t} \mu =\pfreepower{b}{s'}\pfreepower{a}{t'}\mu, \qquad \text{for all } s\in (0,1),\,  t\geq 1.$$
Altogether we have shown the claim for all $s>0$ and $t\geq 1$, as desired.
\end{proof}

By composing $\pfreepower{b}{t+1} $ and $\pfreepower{a}{\tfrac{1}{t+1}} $, we obtain a Belinschi-Nica type semigroup.

\begin{definition}
For $a,b\in \hat{\rr}$, $b\neq a$, $t\geq 0$ and we define the map $ B^{b,a}_t: \text{FID}(a,\hat{\rr})/\{\delta_a\}\to \text{FID}(a,\hat{\rr})/\{\delta_a\}$ such that for all $\mu \in \text{FID}(a,\hat{\rr})/\{\delta_a\}$ we have
$$ B^{b,a}_t\mu := \pfreepower{b}{1+t}\pfreepower{a}{\frac{1}{1+t}}\mu.$$   
\end{definition}

\begin{remark}
Notice that although $ B^{b,a}_t\mu $ is defined for all $t\geq 0$ and $\mu \in \text{FID}(a,\hat{\rr})/\{\delta_a\}$, it is possible that $ B^{b,a}_t\mu =\delta_b $, which then implies that $ B^{b,a}_{t_0}\mu =\delta_b $ for all $t_0\geq t$.
\end{remark}

\begin{remark}
    In general, without assuming that $\mu\in \text{FID}(a,\hat{\rr})$, $ B^{b,a}_t\mu $ is still well defined when $ \pfreepower{a}{\frac{1}{1+t}}\mu $ exists (i.e. $ \mu $ is $(1+t)$-$a$-freely divisible).
\end{remark}

Clearly, the maps $B^{b,a}_t$ form a semigroup.

\begin{proposition}
For $a,b\in \hat{\rr}$, $b\neq a$, the family $\{ B^{b,a}_t \}_{t\geq 0} $ forms a semigroup on $ \text{FID}(a,\hat{\rr})/\{\delta_a\} $ for $t\geq 0$. Namely,
$$B^{b,a}_t \circ B^{b,a}_s =B^{b,a}_{t+s}\qquad \text{for all } t,s \geq 0.$$ 
\end{proposition}

\begin{proof}
A direct computation yields 
$$ B^{b,a}_s B^{b,a}_t=\pfreepower{b}{1+s}\pfreepower{a}{\frac{1}{1+s}} \pfreepower{b}{1+t}\pfreepower{a}{\frac{1}{1+t}} = \pfreepower{b}{1+s}\pfreepower{b}{\frac{1+s+t}{1+s}}\pfreepower{a}{\frac{1+t}{1+s+t}} \pfreepower{a}{\frac{1}{1+t}}  = \pfreepower{b}{1+s+t} \pfreepower{a}{\frac{1}{1+s+t}}  = B^{b,a}_{s+t}, $$
for all $s,t>0$.
\end{proof}

\begin{remark}
Observe that $ B_{t}^{a,b} $ is the inverse of $ B_t^{b,a} $ as for all $\mu$ such that $ \pfreepower{a}{\frac{1}{1+t}}\mu $ exists,
$$ B_{t}^{a,b}B_t^{b,a}\mu = \pfreepower{a}{{1+t}}\pfreepower{b}{\frac{1}{1+t}}\pfreepower{b}{1+t}\pfreepower{a}{\frac{1}{1+t}}\mu = \mu. $$
\end{remark}

\begin{remark}
    In \cite{campbell2024free}, it is observed that the finite free analogy of $\text{FID}(\infty,\hat{\rr})$ are Appell sequences with real roots. That is, a sequence of real rooted polynomials $ p_{n}\in \pol_{n}(\rr) $ such that $ \partial p_n = np_{n-1} $. Indeed, intuitively, since repeated differential is the analogy of $\freepower{t}$, that a polynomial $p_n\in \pol_n(\rr)$ is ‘finite freely infinitely divisible' is saying that it admits real rooted repeated antiderivatives with arbitrary degrees, namely $p_n$ belongs to an Appell sequence with real roots.

    Then one should naturally expect that the finite free analogy of $\text{FID}(a,\hat{\rr})  $ should be a sequence of real rooted polynomials $p_n\in \pol_n(\hat{\rr})$ such that $ \polar{a}p_{n} = c_np_{n-1} $ for some $ c_n\in \rr/\{0\} $. We call such a sequence an \emph{$a$-Appell sequence with real roots}.

    Appell sequences (not necessarily real rooted) can be naturally constructed using Jensen polynomials (see \cite{campbell2024universality}\cite{campbell2024free}): Given an entire function $f(z)= \sum_{k=0}^\infty 
    \frac{\gamma_k}{k!}z^k$, the associated Jensen polynomials is the sequence $ \{J_{n,f}\}_{n=1}^\infty $, defined by
    $$ J_{n,f}(z):= \sum_{k=0}^n \gamma_k \binom{n}{k}z^k = f(x\polar{0})(x-\infty)^n, $$
    which is a $0$-Appell sequence as $ \polar{0} $ commutes with $ f(x\polar{0}) $, and $ 1=(x-\infty)^n \in \pol_n(\hat{\rr})$.
    
    From this, one can then construct the Appell sequence,
    $$ A_{n,f}(z):= z^nJ_{n,f}(1/z) = \sum_{k=0}^n \gamma_k \binom{n}{k}z^{n-k}= f(D)z^n. $$
    These polynomials have the property that $ J_{n,f} $ (or equivalently $A_{n,f}$) has real roots if and only if $f$ belong to the Laguerre-P{\'o}lya class 
    $$ \left\{\left. f(z)= Cz^m e^{cz- \frac{\sigma^2}{2}z^2}\prod_{j=1}^\infty( 1-\alpha_jz )e^{ \alpha_j z }\right|  C,c,\sigma\in \rr,m\in \mathbb{N},\{\alpha_j\}_{j\geq 1}\in \ell^2 \right\}. $$

    In general, for $a\neq b $ in $\hat{\rr}$, we can construct an $a$-Appell sequence
    $$ f((x-a)\polar{a})(z-b)^n, $$
    which is also real rooted when $f$ belong to the Laguerre-P{\'o}lya class.
\end{remark}

%%%%%%%%%%%%%%%%%%%%%%%%%%%%%%%%%%%%%%%%%%%%%%%%%%%%%%%%%%%%%%%%%%%%%%%%%%
\section{Differential equations.}
\label{sec:analytical}
%%%%%%%%%%%%%%%%%%%%%%%%%%%%%%%%%%%%%%%%%%%%%%%%%%%%%%%%%%%%%%%%%%%%%%%%%%%

In this section, we formally compute the PDEs that the Cauchy transform and the $R$-transform of $\pfreepower{a}{t}\mu$ satisfy for a fixed $\mu\in \MM(\rr)$. The goal is to provide intuition on how the roots of polynomials flow under repeated polar differentiation. This also gives a more analytical intuition of the ideas behind the proof of Theorem \ref{thm:commutation.intro}. However, we must emphasize that our results are not rigorous. In order for this approach to be completely rigorous, one should sort out the problems that arise concerning the domain of the functions and the uniqueness of solutions.

Our starting point is the work of Shlyakhtenko and Tao \cite{shlyakhtenko2022fractional} where the free convolution power is studied from the perspective of PDEs over the Cauchy transform. It is then natural to compute what would be the PDE that $ G_{\pfreepower{a}{t}\mu}(z) $ satisfies.

We first rewrite the basic PDE for the Cauchy transform in \cite{shlyakhtenko2022fractional} in terms of $\freepower{t}$.

\begin{lemma}
\label{lem.ST}
Let $ G(z,t) := G_{ \freepower{t}\mu }(z) $ be the Cauchy transform of $\freepower{t}\mu$ for some $\mu$, then:
$$t\partial_t G(z,t) = G(z,t) +\frac{\partial_z G(z,t)}{G(z,t)} \qquad \text{ for } t\geq 1 \text{ and } z \in \cc\setminus\rr. $$

\end{lemma}
\begin{proof}
By \cite[Equation (1.17)]{shlyakhtenko2022fractional}, $ G_{\mu^{\boxplus t}}(z) $ satisfies that
$$ (t\partial_t + z\partial_z)G_{\mu^{\boxplus t}}(z) = \frac{\partial_z G_{\mu^{\boxplus t}}(z) }{G_{\mu^{\boxplus t}}(z)} \qquad \text{ for } t\geq 1 \text{ and } z \in \cc\setminus\rr .$$ 
Since $ \freepower{t}\mu = \dil{1/t}\mu^{\boxplus t}$, we have
$G(z,t)= G_{ \freepower{t}\mu }(z) = tG_{\mu^{\boxplus t}}(tz) $, and the statement follows from a direct computation.
\end{proof}

Using this we can  obtain a PDE that describes $G_{\pfreepower{a}{t}}$. 

\begin{proposition}
\label{prop:pde.cauchy}
    Let $ G_a(z,t)= G_{ \pfreepower{a}{t}\mu }(z) $, we have $$ t\partial_t G_a(z,t)
    = G_a(z,t)+\frac{ G_a(z,t)+(z-a) \partial_z G_a(z,t) }{-1+(z-a)G_a(z,t)}. $$
\end{proposition}
\begin{proof}
By definition of Cauchy transform,
\begin{equation}
\label{aux1}
 G_{T_* \pfreepower{a}{t}\mu}(z) = \int \frac{1}{z-\frac{1}{x-a}} d\pfreepower{a}{t}\mu(x) = \frac{1}{z}-\frac{1}{z^2}G_{\pfreepower{a}{t}\mu}( a+1/z ).
 \end{equation}

On the other hand, recall that if $T$ is the Möbius transform $T(z):=\frac{1}{z-a}$, then $ T_* \pfreepower{a}{t}\mu = F^t T_*\mu $. Using Lemma \ref{lem.ST} we obtain that $G_{T_* \pfreepower{a}{t}\mu}(z)$ satisfies 
\begin{equation}
\label{aux2}
t\partial_tG_{T_* \pfreepower{a}{t}\mu}(z) = G_{T_* \pfreepower{a}{t}\mu}(z)+ \frac{\partial_z G_{T_* \pfreepower{a}{t}\mu}(z)}{G_{T_* \pfreepower{a}{t}\mu}(z)} .
\end{equation}

Combining \eqref{aux1}  and \eqref{aux2}, we obtain the equation
$$ t\partial_t\left( \frac{1}{z}-\frac{1}{z^2}G_{\pfreepower{a}{t}\mu}( a+1/z ) \right) = \frac{1}{z}-\frac{1}{z^2}G_{\pfreepower{a}{t}\mu}( a+1/z ) + \frac{ \partial_z\left[\frac{1}{z}-\frac{1}{z^2}G_{\pfreepower{a}{t}\mu}( a+1/z )\right] }{\frac{1}{z}-\frac{1}{z^2}G_{\pfreepower{a}{t}\mu}( a+1/z )}.$$

Since $G_a(z,t):= G_{\pfreepower{a}{t}\mu}(z)$, the equation becomes
\begin{align*}
    -\frac{1}{z^2}t\partial_t G_a(a+1/z,t) 
    &= -\frac{1}{z^2}G_a(a+1/z,t)+\frac{ \frac{1}{z^2}G_a(a+1/z,t)+\frac{1}{z^3}\partial_zG_a(a+1/z,t) }{1-\frac{1}{z}G_a(a+1/z,t) }.
\end{align*}
The statement then follows from a change of variable.
\end{proof}

Recall that the $R$-transform of a probability measure $\mu$ is defined to be
$$ R_\mu(z) :=G_\mu^{\langle -1\rangle}(z)-\frac{1}{z}, $$
where $ G_\mu^{\langle -1\rangle} $ is the inverse function of $G$. The $R$-transform satisfies the relation $ R_{\mu^{\boxplus t}}(z) = t R_{\mu}(z)  $, implying that $ R_{\freepower{t}\mu}(z) = R_\mu(\tfrac{z}{t}) $.

\begin{proposition}\label{prop:PDEforR}
     Let $ R_a(z,t):=R_{\pfreepower{a}{t}\mu}(z) $ be the $R$-transform of $ \pfreepower{a}{t}\mu $, then at a formal level the following equality holds
     $$  t\partial_t R_a(z,t) = -z\partial_z R_a(z,t) +\frac{ \partial_z R_a(z,t) }{ a-R_a(z,t) }.$$
\end{proposition}
\begin{proof}
    Since $R_a(z,t) = R_{\pfreepower{a}{t} \mu}(z) $, we have
$$ G_a( R_a(z,t)+1/z,t ) = z. $$
Apply $ \partial_z $ and $\partial_t$, we obtain
$ ( \partial_z R_a(z,t)-1/z^2 )\partial_z G_a(R_a(z,t)+1/z,t) = 1 $ and $ (\partial_t R_a(z,t))\partial_z G_a(R_a(z,t)+1/z,t) + \partial_t G_a(R_a(z,t)+1/z,t) = 0$. This implies that
$$\begin{cases}
    \partial_z G_a(R_a(z,t)+1/z,t) &= \frac{1}{\partial_z R_a(z,t)-1/z^2}\\
    \partial_t G_a(R_a(z,t)+1/z,t) &= -\frac{\partial_t R_a(z,t)}{\partial_z R_a(z,t)-1/z^2}.
\end{cases}$$
Therefore, using Proposition \ref{prop:pde.cauchy} we obtain that
$$-\frac{t\partial_t R_a(z,t)}{\partial_z R_a(z,t)-1/z^2}= t\partial_t G_a(R_a(z,t)+1/z,t)=z- \frac{z+ \frac{ R_a(z,t)+1/z-a} {\partial_z R_a(z,t)-1/z^2}}{ 1-z(R_a(z,t)+1/z-a) }$$

The conclusion follows after simplifying.
\end{proof}

We now try to solve this PDE. If we do the change of variables $ \tau =\tfrac{1}{t}$ and $\xi = \tfrac{z}{t},$ and denote $\bar{R}_a(\tau,\xi):=R_a(z,t)$ we obtain

$$ \partial_\tau \bar{R}_a + \frac{\partial_\xi \bar{R}_a}{a-\bar{R}_a} =0. $$
This is a quasilinear PDE, and can be solved using the characteristic lines: for each $ \xi_0 $, $\bar{R}_a$ is constant over the characteristic line $ \left\{ (\tau,\xi): \xi = \xi_0+ \frac{\tau -1 }{a-R_\mu( \xi_0 )} \right\} $.

In particular, $\bar{R}_a$ satisfies the implicit function relation
$$ \bar{R}_a\left( \tau ,\  \xi_0+ \frac{\tau -1 }{a-R_\mu( \xi_0 )}  \right) = R_\mu( \xi_0 ). $$

Or, in terms of $ z $ and $t$,
\begin{equation}
\label{eq.PDE}
R_a\left( t\xi_0+\frac{1 -t }{a-R_\mu( \xi_0 )} ,\ t \right) = R_\mu(\xi_0).     
\end{equation}

We can check that Proposition \ref{prop:PDEforR} is compatible with Theorem \ref{thm:commutation.intro}, in the sense that the Cauchy transform of $ \pfreepower{a}{s} \freepower{t}\mu$ 
and $\freepower{s'}\pfreepower{a}{t'}\mu $ satisfy the same PDE for $s$ and $z$.

\begin{proposition}
    Let $s,t\geq 1$. For any $a\in \rr\cup\{\infty\} $ and $ \mu \in \MM(\rr/\{a\}) $, let $$ \hat{R}(z,s):= R_{ F^{s'}F_a^{t'}\mu }(z) \qquad \text{and} \qquad \tilde{R}(z,s):= R_{ \pfreepower{a}{s} F^t\mu }(z) .$$ Then for a fixed $t$, $ \hat{R} $ and $\tilde{R}$ satisfies the same PDE for $z$ and $s$,
    $$ s\partial_s \tilde{R}(z,s) = -z\partial_z \tilde{R}(z,s) +\frac{ \partial_z \tilde{R}(z,s) }{ a-\tilde{R}(z,s) },\quad s\partial_s \hat{R}(z,s) = -z\partial_z \hat{R}(z,s) +\frac{ \partial_z \hat{R}(z,s) }{ a-\hat{R}(z,s) }. $$
    Here again $s',t'\geq 1$ are determined by the relations $ st=s't'$, and $ s+s'=1+st $.
\end{proposition}

\begin{proof}
By Proposition \ref{prop:PDEforR}, we directly have
\begin{align*}
    s\partial_s \tilde{R}(z,s) = -z\partial_z \tilde{R}(z,s) +\frac{ \partial_z \tilde{R}(z,s) }{ a-\tilde{R}(z,s) }.
\end{align*}
On the other hand, letting $ R_a(z,t):= R_{ \pfreepower{a}{t}\mu }(z) $ we obtain

$$
\hat{R}(z,s)= R_{\dil{\frac{1}{s'}}[F_a^{t'}\mu ]^{\boxplus s'} }(z) = R_{ F_a^{t'}\mu }\left( \frac{z}{s'} \right) = R_a\left( \frac{z}{s'},t' \right) = R_a\left( \tfrac{z}{1+st-s},\tfrac{st}{1+st-s} \right).
$$
This implies 
$$ \partial_z \hat{R}(z,s) = \frac{1}{1+st-s}\partial_zR_a\left( \tfrac{z}{1+st-s},\tfrac{st}{1+st-s} \right), $$
and
\begin{equation}
\label{eq.aux.pde}    
    \partial_s \hat{R}(z,s) = \frac{(1-t)z}{(1+st-s)^2}\partial_zR_a\left( \tfrac{z}{1+st-s},\tfrac{st}{1+st-s} \right)+ \frac{ t}{ (1+st-s)^2} \partial_tR_a\left( \tfrac{z}{1+st-s},\tfrac{st}{1+st-s} \right).
\end{equation}
Finally, using Proposition \ref{prop:PDEforR} (with modified parameters) in the second term of \eqref{eq.aux.pde} we get
\begin{align*}
    & s\partial_s \hat{R}(z,s) = \frac{(s-st)z}{(1+st-s)^2}\partial_zR_a\left( \tfrac{z}{1+st-s},\tfrac{st}{1+st-s} \right)+ \frac{st}{ (1+st-s)^2} \partial_tR_a\left( \tfrac{z}{1+st-s},\tfrac{st}{1+st-s} \right)\\
    &=\left(\tfrac{(s-st)z}{(1+st-s)^2} +\tfrac{1}{ (1+st-s)}\left( -\tfrac{z}{1+st-s}+ \tfrac{1}{a-\hat{R}(z,s)} \right) \right)\partial_zR_a\left( \tfrac{z}{1+st-s},\tfrac{st}{1+st-s} \right)\\
    &= \left( -z +\frac{1}{a-\hat{R}(z,s)} \right) \frac{ 1 }{1+st-s}\partial_zR_a\left( \tfrac{z}{1+st-s},\tfrac{st}{1+st-s} \right)\\
    &=-z\partial_z \hat{R}(z,s) +\frac{ \partial_z \hat{R}(z,s) }{ a-\hat{R}(z,s) },
\end{align*}
and this concludes the proof.
\end{proof}

\begin{remark}
In the previous proposition, we do not know if the solution to the PDE is unique. Actually, we do not even know if there is an appropriate domain where the PDE for $R$ can be defined. If we were able to find an appropriate domain for the PDE and show that the solution is unique, then this would provide another proof of Theorem \ref{thm:commutation.intro}. 
\end{remark}

%%%%%%%%%%%%%%%%%%%%%%%%%%%%%%%%%%%%%%%%%%%%%%%%%%%%%%%%%%%%%%%%%%%%%%%%%%
\section{Examples}
\label{sec:examples}
%%%%%%%%%%%%%%%%%%%%%%%%%%%%%%%%%%%%%%%%%%%%%%%%%%%%%%%%%%%%%%%%%%%%%%%%%%

A large class of polynomials with real roots is contained in the class of hypergeometric polynomials; these were recently studied in connection with finite free probability in \cite{martinez2024hypergeometric}. This class contains several important families, such as Bessel, Laguerre and Jacobi polynomials. These polynomials constitute a rich class of examples in the theory of finite free probability.

\begin{definition}
    \label{def:H.polynomials}
    For $n \geq 1$, pick some parameters
    \begin{equation} \label{eq:condition.hypergeometric}
    a_1, \ldots, a_i \in \rr \setminus \left\{ 0,\tfrac{1}{n},\tfrac{2}{n}, \ldots, \tfrac{n-1}{n}\right\} \text{ and } b_1,\ldots,b_j \in \rr^j.
    \end{equation} 
The hypergeometric polynomial with the given parameters is defined as:
 \[ \HGP{n}{b_1,\ldots,b_j}{a_1,\ldots,a_i} := \sum_{k=0}^n x^{n-k} (-1)^k \binom{n}{k} \frac{\falling{nb_1}{k} \cdots \falling{nb_j}{k}}{\falling{na_1}{k}\cdots \falling{na_i}{k}}, \]
where $\falling{\alpha}{k}:=\alpha(\alpha-1)\cdots (\alpha-k+1)$ denotes the $k$-th falling factorial.
    
To simplify notation, for a tuple $\bm a =(a_1, \ldots, a_i)$ we will write
\[ \falling{\bm a}{k} := \prod_{s=1}^i \falling{a_s}{k} \text{.} \]
Then, for tuples $\bm a = (a_1,\ldots,a_i)$ and $\bm b = (b_1,\ldots,b_j)$ satisfying \eqref{eq:condition.hypergeometric}, we write
\[ \HGP{n}{\bm b}{\bm a} := \sum_{k=0}^n x^{n-k} (-1)^k \binom{n}{k} \frac{\falling{n\bm b}{k} }{ \falling{n \bm a}{k}}. \]
\end{definition}

The limiting distribution of hypergeometric polynomials can be expressed concretely in terms of the $S$-transform.

\begin{notation}
\label{nota:rhoparamseven}
If $\mu\in \MM(\rr_{\geq 0})$ is a measure with $S$-transform of the form
\[ S_{\mu}(z)= \frac{\prod_{r=1}^i (z+a_r) }{\prod_{s=1}^j(z+b_s)} \text{,} \]
for some parameters $a_1,\dots, a_i$ and $b_1,\dots,b_j$, then we say that $\mu$ is an \emph{$S$-rational measure} and denote it by $\SRM{b_1,\dots, b_j}{a_1,\dots, a_i}$.
\end{notation}

\begin{theorem}[{\cite[Theorem 3.9]{martinez2025zeros}, \cite[Corollary 10.8]{arizmendi2024s}}]
    \label{thm:HypLimit}
    Given integers $i,j\geq 0$, consider tuples $\bm A=(A_1,\dots, A_i) \in (\rr\setminus [0,1))^i$ and $\bm B=(B_1,\dots, B_j)\in (\rr\setminus \{0\})^j$. Assume that $\mfp=(p_n)_{n\geq 0}$ is a sequence of polynomials such that
    \begin{equation} \label{sequencePn}
        p_n=\dil{n^{i-j}}\HGP{n}{\bm b_n}{\bm a_n}\in\pols_n(\rr_{> 0}),
    \end{equation}
    where the tuples of parameters $\bm a_n \in \rr^i$ and $\bm b_n\in \rr^j$ have a limit given by
    \begin{equation} \label{assumptionlimits}
        \lim_{n\to\infty} \bm a_n=\bm A, \qquad \text{and} \qquad \lim_{n\to\infty} \bm b_n=\bm B.
    \end{equation}
    Then, as $n$ tends to $\infty$ the empirical root distribution $\meas{p_n}$ converges weakly to the $S$-rational measure $$\SRM{b_1,\dots, b_j}{a_1,\dots, a_i}.$$ 
\end{theorem}

\begin{example}[Laguerre polynomials and free Poisson]
\label{exm:free.poisson}
Let us look at the hypergeometric polynomials with only one upper parameter $\lambda$, better known as the Laguerre polynomials,
$$\HGP{n}{\lambda}{\cdot}:=\sum_{k=0}^n x^{n-k} (-1)^k \binom{n}{k} \falling{n\lambda}{k}.$$

Notice that its $0$-polar derivative is again a Laguerre polynomial:
\begin{align*}
\polar{0}\HGP{n}{\lambda}{\cdot}&=\sum_{k=0}^n x^{n-k} (-1)^k \binom{n}{k} \falling{n\lambda}{k} (n-(n-k)) \\
&=-n^2\lambda \sum_{k=0}^{n-1} x^{n-1-(k-1)} (-1)^{k-1} \binom{n-1}{k-1} \falling{n\lambda -1}{k-1} \\
&=-n(n\lambda) \HGP{n-1}{\frac{n}{n-1}(\lambda-1)+1}{\cdot}.
\end{align*}

Thus, after differentiating $n-m$ times we obtain
$$\polars{0}{m}{n}\HGP{n}{\lambda}{\cdot}=c \, \HGP{m}{\frac{n}{m}(\lambda-1)+1}{\cdot},$$ 
where $c=:(-1)^{n-m}\falling{n}{n-m} \falling{n\lambda}{n-m}$ is just a constant.

By letting $n$ and $m$ tend to $\infty$ while $\frac{n}{m}$ tends to $t$, we obtain a relation for the free Poisson (or Marchenko-Pastur) distribution $\pi_\lambda:=\SRM{\lambda}{\cdot}$. Indeed, using Theorems \ref{thm:asymptotic.polar.diff} and  \ref{thm:HypLimit} we obtain that
\begin{align*}
F_0^t\left( \pi_\lambda \right) &= \lim_{n\to\infty}\meas{\polars{0}{m}{n} \dil{\frac{1}{n}} \HGP{n}{\lambda}{\cdot} }\\
&= \lim_{n\to\infty} \meas{ \dil{\frac{m}{n}} \dil{\frac{1}{m}}  \HGP{m}{\frac{n}{m}\lambda-\frac{n}{m}+1}{\cdot} }\\
&= \dil{\frac{1}{t}}\,  \pi_{t\lambda-t+1}.
\end{align*}

We are not aware of a similar result in the literature. However, it can also be directly proved using cumulants. Indeed, let us denote by $\pi_\lambda^{-1}$ the reciprocal of the free Poisson distribution, obtained by the pushforward of $\pi_\lambda$ along the reciprocal map. In other words, if $ X_\lambda \sim \pi_\lambda $ is a free Poisson random variable with intensity $\lambda>1$, then the multiplicative inverse is $X_\lambda^{-1}\sim \pi_\lambda^{-1}$. The cumulant formula in \cite[Remark 3.2]{Szpojankowski14} asserts that the free cumulants of the reciprocal are given by
$$ R_n\left( \pi_\lambda^{-1} \right) = (\lambda-1)^{1-2n}C_{n-1}, $$
where $C_n:= \frac{1}{n+1}\binom{2n}{n}$ is the $n$-th Catalan number.
Therefore the free cumulants of the free convolution power of the reciprocal are 
$$ R_n\left( \freepower{t} (\pi_\lambda^{-1})\right) = t^{1-n}( \lambda-1 )^{1-2n}C_{n-1} = t^{n}(t\lambda -t)^{1-2n}C_{n-1}=R_n\left( \dil{t}\pi^{-1}_{t\lambda -t +1}\right).$$
Applying the reciprocal map on both sides we conclude that
\begin{equation}
\label{eq.polar.0.poisson}
    F_0^t( \pi_{\lambda} ) = \dil{1/t} \pi_{ {t\lambda -t +1}  }.
\end{equation}

Notice that this result means that the family of free Poisson distributions $(\pi_\lambda)_{\lambda > 1}$ is invariant under both semigroups $\freepower{t}$ and $ \pfreepower{0}{t}$ (up to a constant dilation). Where the former is a well-known fact in free probability, $\freepower{t} \pi_\lambda= \dil{1/t} \pi_{t\lambda},$
that follows from checking the cumulants. In the polynomial framework this corresponds to the fact that the family of Laguerre polynomials is preserved under differentiation, $\polar{\infty}$, and 0-polar differentiation, $\polar{0}$.

To obtain a similar relation for $\pfreepower{a}{t}$ with $a\in \rr$, we can simply use Lemma \ref{lem:FintertwinedbyMobius} with a shift by $a$:
\begin{equation}
\label{eq.polar.a.poisson} 
F_a^t ( \delta_a \boxplus \pi_{\lambda} ) =F_a^t ( \shift{a} \pi_{\lambda} ) = \shift{a} F_0^t\pi_{\lambda} = \delta_a \boxplus \dil{1/t} \pi_{ {t\lambda -t +1} }.
\end{equation}
The analogue for the Laguerre polynomials follows from Proposition \ref{prop:FintertwinedbyMobius},
$$\polars{a}{m}{n}\left(\HGP{n}{\lambda}{\cdot}(x+a) \right)=c \, \HGP{m}{\frac{n}{m}(\lambda-1)+1}{\cdot}(x+a).$$

We can also compute the action of the semigroup $B_{t}^{\infty,0}$ on $\pi_\lambda$:
$$  B^{\infty,0}_t\pi_{\lambda} =  \freepower{1+t}\pfreepower{0}{\tfrac{1}{1+t}}( \pi_{\lambda} )=\freepower{1+t}  \dil{t+1} \pi_{ \tfrac{\lambda-1}{t+1}+1 } = \pi_{ \lambda+t}. $$
\end{example}

\begin{remark}[Using differential equations] 
Notice that using \eqref{eq.PDE} we can derive non-rigorously the result in Equation \eqref{eq.polar.a.poisson}. Indeed, if $ \pi_\lambda $ denotes the free Poisson distribution with intensity $\lambda$, then its $R$-transform is given by
$$ R_{ \pi_\lambda }(z) = \frac{\lambda}{1-z}.$$
So after some computations we obtain that the following equality holds in some domain:
$$ R_{F_a^t ( \delta_a \boxplus \pi_{\lambda} )}(y) = a+\frac{\lambda}{ 1- \frac{\lambda y-t+1}{t\lambda - (t-1)}  } = a+\frac{  t\lambda-(t-1) }{t-y}= R_{\delta_a \boxplus \dil{1/t} \pi_{ t\lambda -t +1}  }(y). $$
This is compatible with relation \eqref{eq.polar.a.poisson}, and if one can show that the equality is true in the correct domain, then this would provide a new proof of the relation.
\end{remark}

\begin{example}[Hypergeometric polynomials and $S$-rational measures] 
The previous example can be easily generalized to a hypergeometric polynomial with arbitrary number of parameters.

Notice that the derivative of a hypergeometric polynomial is again hypergeometric:
$$ \partial \HGP{n}{\bm b}{\bm a}  =  n\sum_{k=0}^{n-1} x^{n-k-1} (-1)^k \binom{n-1}{k} \frac{\falling{nb_1}{k} \cdots \falling{nb_j}{k}}{\falling{na_1}{k}\cdots \falling{na_i}{k}} = n\HGP{n-1}{\tfrac{n}{n-1}\bm b}{\tfrac{n}{n-1}\bm a}.$$
Similarly, the $0$-polar derivative of a hypergeometric polynomials is also hypergeometric:
\begin{align*}
\polar{0}\HGP{n}{\bm b}{\bm a}&=\sum_{k=0}^n x^{n-k} (-1)^k \binom{n}{k} \frac{\falling{n\bm b}{k}}{\falling{n\bm a}{k}} (n-(n-k)) =-n\frac{(nb_1)(nb_2)\cdots(nb_j)}{(na_1)\cdots(na_i)}\HGP{n-1}{\frac{n\bm b-1}{n-1}}{\frac{n\bm a-1}{n-1}}.
\end{align*}

Thus, after differentiating $n-m$ times we obtain
$$ \diff{m}{n}\HGP{n}{\bm b}{\bm a}=c \, \HGP{m}{\frac{n}{m}\bm b}{\frac{n}{m}\bm a}, $$
and
$$\polars{0}{m}{n}\HGP{n}{\bm b}{\bm a}=c' \, \HGP{m}{\frac{n}{m}\bm b-\frac{n}{m}+1}{\frac{n}{m}\bm a-\frac{n}{m}+1},$$ 
where $ c= \falling{n}{n-m}$ and $c'=(-1)^{n-m}\frac{ \falling{n}{n-m}\falling{n\bm b}{n-m}}{\falling{n\bm a}{n-m}}$ are just constants.

Letting $n$ and $m$ tend to $\infty$ while $\frac{n}{m}$ tends to $t$, and using Theorems \ref{thm:asymptotic.polar.diff} and  \ref{thm:HypLimit}  yields
\begin{align*}
\freepower{t}\left( \SRM{\bm b}{\bm a} \right) &= \lim_{n\to\infty}\meas{\diff{m}{n} \dil{n^{i-j}} \HGP{n}{\bm b}{\bm a} }\\
&= \lim_{n\to\infty} \meas{ \dil{(\frac{m}{n})^{i-j}} \dil{m^{i-j}}  \HGP{m}{\frac{n}{m}\bm b}{\frac{n}{m}\bm a} }\\
&= \dil{t^{j-i}}\,  \SRM{t\bm b}{t\bm a}.
\end{align*}
and similarly,
\begin{align*}
\pfreepower{0}{t}\left( \SRM{\bm b}{\bm a} \right) &= \lim_{n\to\infty}\meas{\polars{0}{m}{n} \dil{n^{i-j}} \HGP{n}{\bm b}{\bm a} }\\
&= \lim_{n\to\infty} \meas{ \dil{(\frac{m}{n})^{i-j}} \dil{m^{i-j}}  \HGP{m}{\frac{n}{m}\bm b-\frac{n}{m}+1}{\frac{n}{m}\bm a-\frac{n}{m}+1} }\\
&= \dil{t^{j-i}}\,  \SRM{t\bm b-t+1}{t\bm a-t+1}.
\end{align*}
In particular, we have
$$ B_t^{\infty,0}\left( \SRM{\bm b}{\bm a} \right) =   \freepower{1+t}\pfreepower{0}{\tfrac{1}{1+t}}\left( \SRM{\bm b}{\bm a} \right) = \freepower{1+t}\dil{{(1+t)}^{i-j}}\,\SRM{\frac{1}{t+1}(\bm b-1)+1}{\frac{1}{t+1}(\bm a-1)+1} =  \SRM{\bm b+t}{\bm a+t}. $$

Once we computed the 0-polar derivative, one can extrapolate this to $\alpha$-polar derivative for an arbitrary $\alpha\in\rr$ by translating the polynomial:
$$\polars{\alpha}{m}{n}\left( \shift{\alpha} \HGP{n}{\bm b}{\bm a} \right)=\shift{\alpha} \left(c' \, \HGP{m}{\frac{n}{m}\bm b-\frac{n}{m}+1}{\frac{n}{m}\bm a-\frac{n}{m}+1} \right).$$ 
We can provide a specific example of $1$-polar derivative using the theory of hypergeometric polynomials. Indeed, using \cite[Equation (12)]{grinshpan2013generalized} and \cite[Equation (84)]{martinez2024hypergeometric} we obtain
$$\shift{-1} \HGP{n}{b}{a} (x)=(x+1)^n\boxplus \HGP{n}{b}{a} = c\, \dil{-1}\HGP{n}{a-b}{a}(x).$$

Therefore,
$$\polars{1}{m}{n}\HGP{n}{b}{a}=c \,\shift{1}  \, \dil{-1} \HGP{m}{\frac{n}{m} (a-b)-\frac{n}{m}+1}{\frac{n}{m} a-\frac{n}{m}+1}  = c_1 \HGP{n}{\frac{n}{m} b}{\frac{n}{m} a-\frac{n}{m}+1}.$$

Again, let $n$ and $m$ tend to $\infty$ while $\frac{n}{m}$ tends to $t$, and using Theorems \ref{thm:asymptotic.polar.diff} and  \ref{thm:HypLimit}  yields
\[ \pfreepower{1}{t}\left( \SRM{ b}{ a} \right) =  \SRM{t b}{ ta-t+1}. \]

This relation represents a symmetric property of the free beta distributions, as studied in \cite[Section 3.4]{yoshida2020remarks}. Indeed, $ \SRM{ b}{ a}$ is called a free beta distribution in \cite{yoshida2020remarks} in the following sense: Suppose $ a>b>1 $, $X\sim \pi_{a-b}$ and $Y\sim \pi_{b}$ be two freely independent self-adjoint variables with free Poisson distributions, then we have
$$ Z_1:=Y(X+Y)^{-1}\sim \SRM{ b}{ a},\quad Z_2:=X(X+Y)^{-1}\sim \SRM{ a-b}{ a}.$$
Now, considering $T(z) = 1-z$, we have than $T(Z_1) = Z_2 $, thus 
$$ \pfreepower{0}{t}T_*\left(\SRM{ b}{ a}\right) = \pfreepower{0}{t}\left(\SRM{ a-b}{ a}\right) = \SRM{ t(a-b)-t+1}{ ta-t+1} .$$
The same argument shows that 
$$T_{*}\SRM{ t(a-b)-t+1}{ ta-t+1} = \SRM{ ta-t+1-(t(a-b)-t+1)}{ ta-t+1} = \SRM{ tb}{ta-t+1}.$$
Therefore, we again obtain $  \pfreepower{1}{t}\left( \SRM{ b}{ a} \right)= T_*\pfreepower{0}{t}T_*\left(\SRM{ b}{ a}\right) =  \SRM{t b}{ ta-t+1}. $
\end{example}

\begin{example}[Cauchy distribution is invariant under $ \pfreepower{a}{t} $]
The Cauchy distribution $\nu$ has density given by $$ d\nu := \frac{1}{\pi}\frac{1}{1+x^2}dx.$$ Its Cauchy transform is given by $$G_{\nu}(z) = \begin{cases}
    \frac{1}{z+i},z\in \cc^+\\
    \frac{1}{z-i},z\in \cc^-
\end{cases}$$

Moreover, one has that $ \nu^{\boxplus t} = \dil{t}\nu$ for $t\geq 1$. Equivalently, the Cauchy distribution is invariant under the semigroup  $\freepower{t}$:
$$ \freepower{t}\nu= \nu \qquad \text{for } t\geq 1. $$ 
In fact, we will show that for every $a\in \hat{\rr}$ the Cauchy distribution is invariant under $\pfreepower{a}{t}$:
$$ \pfreepower{a}{t}\, \nu = \nu \qquad \text{for } t\geq 1. $$ 

Let $ T(z) := \frac{1}{z-a} $, a straightforward computation shows that for $ z\in \cc^+ $
$$ G_{T_* \nu}(z) = \int_\rr\frac{1}{z-\frac{1}{x-a}}d\nu(x)=  \frac{1}{z}-\frac{G_{\nu}\left( \frac{1}{z}+a\right)}{z^2}=\frac{1}{z}- \frac{1}{z^2}\frac{1}{\frac{1}{z}+a  - i} 
= \frac{1}{z+\frac{1}{a-i}}.$$
Using that $ \frac{1}{a-i} = \frac{a+i}{a^2+1} $, we obtain
$$ T_*(\nu) = \shift{\frac{-a}{a^2+1}} \dil{\frac{1}{a^2+1}}\nu .$$
Then, since $F_t$ commutes with dilations and shifts (see Lemma \ref{lem:freepower.commutes.mob.infty}), we conclude
\begin{align*}
    \pfreepower{a}{t}\nu &= T_*^{-1}\freepower{t}T_*\nu = T_*^{-1}\freepower{t} \left(  \shift{\frac{-a}{a^2+1}} \dil{\frac{1}{a^2+1}}\nu \right)=T_*^{-1}  \shift{\frac{-a}{a^2+1}} \dil{\frac{1}{a^2+1}}\freepower{t}\nu \\
    &= T_*^{-1}  \shift{\frac{-a}{a^2+1}} \dil{\frac{1}{a^2+1}}\nu=T_*^{-1}T_*\nu = \nu.
\end{align*}

Since $ \pfreepower{a}{t}\nu = \nu $ for all $a\in \hat{\rr}$ and $t>0$, we also have $ B_{t}^{a,b}\nu = \nu $ for all $a,b\in \hat{\rr}$.

The analogue of the Cauchy distribution $\nu$ in the framework of polynomials was devised recently by Campbell \cite[Theorem 2.4]{campbell2024free}. In fact, it is observed that $\nu$ can actually be realized as the asymptotic root distribution of the Appell sequence $ \{C_n\}_{n\geq 1} $ (i.e. $ C_{n-1}=\partial C_n $), defined as
$$ C_n(x) := \cos(\partial)x^n =\sum_{k=0}^{\left\lfloor n/2 \right\rfloor} (-1)^k \binom{n}{2k} x^{n-2k}. $$

We will consider each $C_n$ as a polynomial with precise and formal degree $n$.

We now apply $\polar{0}$ to the Appell sequence $C_n$:
\begin{align*}\polar{0}C_n(x) = \sum_{k=1}^{\left\lfloor n/2 \right\rfloor} (-1)^k 2k\binom{n}{2k} x^{n-2k}.
\end{align*}
Note that since $C_n$ is centered, $\polar{0}C_n$ has a root at $\infty$, and therefore should be considered as a polynomial with formal degree $n-1$ but precise degree $n-2$. Apply $\polar{0}$ again, we then obtain
\begin{align*}
    \polar{0}^2C_n(x)&= \sum_{k=1}^{\left\lfloor n/2 \right\rfloor} (-1)^k 2k(2k-1)\binom{n}{2k} x^{n-2k}\\  &= n(n-1)\sum_{k=0}^{\left\lfloor n/2 \right\rfloor-1} (-1)^{k+1} \binom{n-2}{2k} x^{n-2-2k}\\&= -n(n-1)C_{n-2}(x)
\end{align*}
Therefore, up to a constant, we have $ \polar{0}^2C_n = C_{n-2}=\partial^2C_n $ (although $ \polar{0}C_n \neq \partial C_n  $). In general, when $0\neq a\neq\infty$, the polar derivative $\polar{a}^kC_n$ does not have a straightforward relation with $ C_{n-k} $. However, since $ \pfreepower{a}{t}\nu =\nu$, asymptotically $ \polar{a}^{\lfloor\tfrac{n}{s}\rfloor}C_n $ still converges to $ \nu $.

\end{example}

\subsection*{Acknowledgments}

We thank Michael Anshelevich, Octavio Arizmendi, Andrew Campbell, Jorge Garza--Vargas, Brian Hall and Jonas Jalowy for some insightful comments during the preliminary versions of this paper. Part of this work was carried out while the first author was affiliated with Texas A\&M University. 

\bibliographystyle{alpha}
\bibliography{References}

\end{document}